\documentclass[a4paper,11pt]{article}

\usepackage[utf8]{inputenc}
\usepackage[english]{babel} 
\usepackage{latexsym}
\usepackage[authoryear]{natbib}
\usepackage{float}
\usepackage{graphicx}
\usepackage{amsmath,amsthm,amssymb,mathtools, mathrsfs} 
\usepackage{enumitem}
\usepackage{layout} 
\usepackage{csquotes}
\usepackage[left=2.7cm,top=4cm,right=2.7cm,bottom=4cm]{geometry} 
\usepackage{framed}
\usepackage[labelfont=sc]{caption}
\usepackage{changepage} 
\usepackage{xcolor}
\usepackage{booktabs,tabularx} 
\usepackage{titling}
\usepackage[capitalize,nameinlink]{cleveref} 

\newcommand{\AVT}{\textnormal{AVAR}(\hat{\theta}_n)}
\newcommand{\AVE}{\widehat{\textnormal{AVAR}}(\hat{\theta}_n)}

\newcommand{\abs}[1]{\left\lvert #1 \right\rvert}
\newcommand{\cc}{\cdot}
\newcommand{\ccs}{\hspace*{0.1cm} \cc \hspace*{0.1cm}}
\newcommand{\dd}{\hspace*{0.05cm}\mathrm{d}}
\newcommand{\given}{\hspace*{0.1cm} \middle\vert \hspace*{0.1cm}}
\newcommand{\norm}[1]{\left\lVert #1 \right\rVert}

\newcommand{\myitem}{\item[$\cc$]}

\newcommand{\R}{\mathbb{R}}

\newcommand{\A}{\mathcal{A}}

\newcommand{\C}{\mathcal{C}}
\newcommand{\D}{\mathcal{D}}
\newcommand{\F}{\mathcal{F}}

\newcommand{\M}{\mathcal{M}}
\newcommand{\N}{\mathcal{N}}
\newcommand{\V}{\mathcal{V}}

\newcommand{\U}{\mathnormal{U}}

\newcommand{\Dnull}{\mathscr{D}_0}
\newcommand{\Enull}{\mathbb{E}_0}
\newcommand{\Gnull}{\mathcal{L}_0}
\newcommand{\Pnull}{\mathbb{P}_0}
\newcommand{\Vnull}{\mathbb{V}{\rm ar}_0}

\newcommand{\cPnull}{\xrightarrow{\Pnull}}
\newcommand{\cDnull}{\xrightarrow{\Dnull}}

\newcommand{\ET}{\mathbb{E}_\theta}
\newcommand{\GT}{\mathcal{L}_\theta}
\newcommand{\PT}{\mathbb{P}_\theta}
\newcommand{\VT}{\mathbb{V}{\rm ar}_\theta}

\newcommand{\FSH}{\mathscr{H}^2_\theta}
\newcommand{\FSL}{\mathscr{L}^2(\mu_\theta)}
\newcommand{\FSLnull}{\mathscr{L}^2_0(\mu_\theta)}

\newcommand{\PTH}{\mathbb{P}_{\hat{\theta}_n}}
\newcommand{\PTG}{\tilde{\mathbb{P}}_{\theta,\gamma}}
\newcommand{\ETG}{\tilde{\mathbb{E}}_{\theta,\gamma}}

\newcommand{\vtwo}[2]{\left(\begin{array}{c} #1 \\ #2 \end{array}\right)}

\newcommand{\mtwo}[4]{\left(\begin{array}{cc} #1 & #2 \\ #3 & #4 \end{array}\right)}

\title{Prediction-based estimation for diffusion models with high-frequency data}

\date{}

\begin{document}
	\maketitle

\vspace{-28mm}

\begin{center}\Large{Emil S. J\o rgensen\footnotemark[1]
      \footnotemark[2] \hspace{10mm} 
    Michael S\o rensen\footnotemark[1] \footnotemark[3]}\end{center} 

\footnotetext[1]{Deptartment of Mathematical Sciences, University of
 Copenhagen, Universitetsparken 5, DK-2100 Copenhagen {\O}, Denmark}

\footnotetext[2]{IKEA Group, {\"A}lmhultsgatan 2, 215 86 Malm{\"o},
  Sweden} 

\footnotetext[3]{Corresponding author, michael@math.ku.dk}

\vspace{8mm}

	\begin{adjustwidth}{1.6cm}{1.6cm}

\footnotesize{

\textsc{Abstract.}
This paper obtains asymptotic results for parametric inference using predic\-tion-based estimating functions when the data are high frequency observations of a diffusion process with an infinite time horizon. Specifically, the data are observations of a diffusion process at $n$ equidistant time points $\Delta_n i$, and the asymptotic scenario is $\Delta_n \to 0$ and $n\Delta_n \to \infty$.  For a useful and tractable  classes of prediction-based estimating functions, existence of a consistent estimator is proved under standard weak regularity conditions on the diffusion process and the estimating function. Asymptotic normality of the estimator is established under the additional rate condition $n\Delta_n^3 \to 0$. The prediction-based estimating functions are approximate martingale estimating functions to a smaller order than what has previously been studied, and new non-standard asymptotic theory is needed.
A Monte Carlo method for calculating the asymptotic variance of the estimators is proposed. \\

\noindent
\textbf{Keywords:} \textit{Diffusion process, high-frequency data, infinitesimal generator, potential operator, parametric inference, prediction-based estimating function, $\rho$-mixing.}

}

\end{adjustwidth}

	\section{Introduction}\label{sec:introduction}

Diffusion processes are often used to model stochastic dynamical systems. An especially successful application area is finance. These processes are defined in continuous time, but for most applications the system is only observed at discrete time points, so statistical methods for discretely observed diffusion processes is a very active area of research. In particular, the availability of high-frequency data has generated considerable interest in the asymptotic behaviour of estimators and test statistics as the time between consecutive observations tends to zero.

In this paper, we study parametric inference for diffusion models that satisfy a stochastic differential equation of the form
\begin{equation}\label{eq:X}
dX_t = a(X_t;\theta)dt + b(X_t;\theta)dB_t,
\end{equation}
where $(B_t)$ is standard Brownian motion, and the known functions $a$ and $b$ depend on a statistical parameter $\theta \in \Theta \subseteq \R^d$ to be estimated. We suppose that $(X_t)$ takes values in an open interval $(l,r) \subseteq \R$ and has a invariant distribution $\mu_\theta$. Moreover, $(X_t)$ is assumed to be stationary under the probability measure $\PT$, i.e.\ $X_0 \sim \mu_\theta$. Let the data be a single discretisation 
\begin{equation*}
X_0,X_{t^n_1},\ldots,X_{t^n_n},
\end{equation*}
where the observation times are deterministic and equidistant, i.e. $t^n_i=i\Delta_n$ for some $\Delta_n>0$. To enable consistent estimation of both drift and diffusion parameters, we consider the ergodic high-frequency sampling scenario
\begin{equation}\label{eq:HF}
n \to \infty, \hspace*{0.5cm} \Delta_n \to 0, \hspace*{0.5cm} n \cc \Delta_n \to \infty,
\end{equation}
where the time horizon $T_n=n\Delta_n$ tends to infinity with the number of observations.

Estimators are defined and studied within the framework of the prediction-based estimating functions, proposed by \cite{pbef-2000, pbef-2011} as a versatile estimation framework, not least for non-Markovian diffusion-type models. They generalize the martingale estimating functions introduced by \cite{mef-1995}. We show that the estimating functions considered in this paper are not approximate martingale estimating functions as defined in \cite{eed-2017}. However, for a two-dimensional predictor space, they are approximate martingale estimating functions of a smaller order than what has previously been studied, namely of order $\Delta_n$ rather than $\Delta_n^k$ for $k \geq 2$. We can still prove existence of consistent estimators, and by applying non-standard limit theory we establish asymptotically normality under mild regularity conditions and the additional rate assumption $n\Delta_n^3 \to 0$. 

Examples from finance and simulation studies as well as more details on the theory and implementation issues can be found in \cite{phd}.

Parametric estimation for discretely observed diffusion processes has been investigated in many papers in the econometrics and statistics literature. Since exact maximum likelihood estimation is untractable for most diffusion models used in practice, a wide range of alternative methods have been proposed and applied successfully. The Markov property of diffusions enables many types of quasi-likelihood, including contrast functions (\cite{asd}, \cite{edp}, \cite{edc-1993}, \cite{mmp}, \cite{eed-1997}), estimating functions (\cite{mef-1995}, \cite{mef-1999}, \cite{sef}, \cite{sm}, \cite{ehf}), likelihood expansions (\cite{edc-1986}, \cite{mle}, \cite{cde}), Markov-chain Monte Carlo (\cite{lid}, \cite{dmf}, \cite{imh}) and simulated likelihood (\cite{eel}, \cite{mce}, \cite{bladtfinchsorensen}).

There is also a well developed literature on nonparametric estimation of the drift and diffusion coefficients from discrete time data. The problem was studied by \cite{npd}, \cite{smd}, \cite{aed}, \cite{esd} and \cite{ped} under the assumption of strict stationarity. Estimation for nonstationary, recurrent diffusion processes was considered by \cite{ned}. Estimation of the diffusion coefficient with high-frequency observations on a finite time horizon was investigated by \cite{ewm}, \cite{edo}, \cite{aed,lpe}, \cite{ked} and \cite{nev}. \cite{nme} gives an excellent survey of nonparametric estimation with an extensive list of references.

The structure of the paper is as follows. In Section \ref{sec:preliminaries} we present the general notation used in the paper, define a tractable class of prediction-based estimating functions, and formulate our general assumption on $(X_t)$. Section \ref{sec:limits} is devoted to limit theorems for functionals $V_n(f) = n^{-1} \sum_{i=1}^n f(X_{t^n_{i-1}})$ and, in particular, a central limit theorem for $f$ belonging to a large class of functions. The variance of the gaussian limit law involves the potential of $f$, which is considered in some detail. Asymptotic results are provided in Section \ref{sec:asymptotics}. In Section \ref{sec:avar} we propose Monte Carlo methods for determining the asymptotic variances obtained in Section \ref{sec:asymptotics}. All proofs are deferred to Appendix~A, and Appendix~B contains some auxiliary results needed in the proofs.

	\section{Preliminaries}\label{sec:preliminaries}

In this section we introduce the notation used throughout the paper, define a tractable class of prediction-based estimating functions, recall some core notions from probability theory, and formulate our main assumptions on the diffusion model $(X_t)$ and the parameter space $\Theta$ for the asymptotic theory.

\subsection{Notation}

Our general notation is as follows:

\begin{enumerate}
\item The parameter of interest $\theta \in \Theta \subseteq \R^d$ for $d \geq 1$. We denote the true parameter by $\theta_0$.
\item We denote the state space of $X$ by $(S,\mathscr{B}(S))$ and assume throughout that $S$ is an open interval in $\R$, i.e. $S=(l,r)$ for $-\infty \leq l < r \leq \infty$, endowed with its Borel $\sigma$-algebra $\mathscr{B}(S)$.
\item The invariant distribution is denoted by $\mu_\theta$. For short, we write $\mu_\theta(f)=\int_S f(x) \mu_\theta(dx)$ for functions $f:S \to \R$, and we denote the canonical norm on $\FSL$ defined by $\norm{f}_2 = \mu_\theta(f^2) ^{1/2}$.
\item For random variables $Y$ and $Z$ defined on a probability space $(\Omega,\F,\mathbb{P})$, we write $Y \leq_C Z$ if there exists a constant $C>0$ such that $Y \leq C \cc Z$, $\mathbb{P}_{\theta_0}$-almost surely. We sometimes use a similar notation for real functions.
\end{enumerate}

To define some function spaces of interest, we say that $f:S \times \Theta \to \R$ is of \emph{polynomial growth in $x$} if for every $\theta \in \Theta$ there exists a constant $C_\theta>0$ such that, $|f(x;\theta)| \leq C_\theta (1 + |x|^{C_\theta})$ for $x \in S$.

\begin{enumerate}[resume]
\item We denote by $\C^{j,k}_p(S \times \Theta)$, $j,k \geq 0$, the class of real-valued functions $f(x;\theta)$ satisfying that
\begin{itemize}
\myitem $f$ is $j$ times continuously differentiable w.r.t. $x$;
\myitem $f$ is $k$ times continuously differentiable w.r.t. $\theta_1,\ldots,\theta_d$;
\myitem $f$ and all partial derivatives $\partial_x^{j_1}\partial_{\theta_1}^{k_1} \cdots \partial_{\theta_d}^{k_d}f$, $j_1 \leq j$, $k_1+\cdots+k_d \leq k$, are of polynomial growth in $x$.
\end{itemize}
We define $\C^j_p(S)$ analogously as a class of functions $f:S \to \R$.
\item For use in the appendices, $R(\Delta,x;\theta)$ denotes a generic function such that
\begin{equation}\label{eq:R}
|R(\Delta,x;\theta)| \leq F(x;\theta),
\end{equation}
where $F$ is of polynomial growth in $x$. We sometimes write $R_0(\Delta,x;\theta)$ to emphasize that the remainder term $R(\Delta,x;\theta)$ also depends on the true parameter $\theta_0$.
\end{enumerate}

\subsection{Prediction-based estimating functions}
\label{ssec:predicestfct}

The general theory of prediction-based estimating functions was developed by \cite{pbef-2000} and later extended in \cite{pbef-2011}. In this paper we consider estimating functions of the general form
\begin{equation}\label{eq:pbef}
G_n(\theta) = \sum_{i=q}^n \sum_{j=1}^N \pi_{i-1,j} \left[f_j(X_{t^n_i})-\breve{\pi}_{i-1,j}(\theta)\right],
\end{equation}
where $\{f_j\}_{j=1}^N$ is a finite set of real-valued functions in $\FSL$ and for every $j \in \{1,\ldots,N\}$, $\breve{\pi}_{i-1,j}(\theta)$ denotes the orthogonal $\FSL$-projection of $f_j(X_{t^n_i})$ onto a finite-dimensional subspace
\begin{equation}\label{eq:pred.space}
\mathcal{P}_{i-1,j} = \text{span}\left\{1,f_j\left(X_{t^n_{i-1}}\right),\ldots,f_j\left(X_{t^n_{i-{q_j}}}\right)\right\} \subset \FSL
\end{equation}
for a fixed $q_j \geq 0$. The coefficients $\pi_{i-1,j}$ are $d$-dimensional column vectors with entries belonging to $\mathcal{P}_{i-1,j}$.

The collection of subspaces $\{\mathcal{P}_{i-1,j}\}_{ij}$ are referred to as \emph{predictor spaces}. In this sense, what we predict are values of $f_j(X_{t^n_i})$ for each $i \geq q:=\max_{1 \leq j \leq N} q_j$. Most prediction-based estimating functions applied in practice are of this particular form; see e.g. \cite{pbef-2000} for applications to discretized stochastic volatility models, and \cite{iid-2004} for the case of integrated diffusions.

Since the predictor space $\mathcal{P}_{i-1,j}$ is closed, the $\FSL$-projection of $f_j(X_{t^n_i})$ onto $\mathcal{P}_{i-1,j}$ is well-defined and uniquely determined by the normal equations
\begin{equation}\label{eq:normal}
\ET\left(\pi\left[f_j(X_{t^n_i})-\breve{\pi}_{i-1,j}(\theta)\right]\right) = 0,
\end{equation}
for all $\pi \in \mathcal{P}_{i-1,j}$; see e.g. \cite{rca}. Here and in everything that follows, $\ET(\ccs)$ denotes expectation w.r.t. the probability measure $\PT$. 
Note that \eqref{eq:normal} implies that the estimating function \eqref{eq:pbef} is unbiased, i.e.
\[
\ET\left(G_n(\theta)\right) = 0.
\]

By restricting ourselves to predictor spaces of the form \eqref{eq:pred.space}, as well as only diffusion models $(X_t)$ that are stationary under $\PT$, the orthogonal projection $\breve{\pi}_{i-1,j}(\theta) = \breve{a}_n(\theta)_j^T Z_{i-1,j}$ where
\begin{equation}\label{eq:Z}
Z_{i-1,j}=\left(1,f_j\left(X_{t^n_{i-1}}\right),\ldots,f_j\left(X_{t^n_{i-{q_j}}}\right)\right)^T
\end{equation}
and $\breve{a}_n(\theta)_j^T$ is the unique $(q_j+1)$-dimensional coefficient vector
\begin{equation*}
\breve{a}_n(\theta)_j^T = \left(\breve{a}_n(\theta)_{j0},\breve{a}_n(\theta)_{j1}\ldots,\breve{a}_n(\theta)_{jq_j}\right)
\end{equation*}
determined by the moment conditions
\begin{equation}\label{eq:moments}
\ET\left[Z_{q_j-1,j} f_j\left(X_{t^n_{q_j}}\right)\right] - \ET\left[Z_{q_j-1,j}Z_{q_j-1,j}^T\right]\breve{a}_n(\theta)_j = 0.
\end{equation}
Note that in the simplest case where $q_j=0$, $\mathcal{P}_{i-1,j}=\text{span}\{1\}$ and it follows immediately from the normal equations \eqref{eq:normal} that $\breve{\pi}_{i-1,j}(\theta) = \mu_\theta(f_j)$.

We obtain an estimator by solving the estimating equation $G_n(\theta)=0$, and we call an estimator $\hat{\theta}_n$ a \emph{$G_n$-estimator} if
\[
\mathbb{P}_{\theta_0} (G_n(\hat{\theta}_n)=0) \to 1
\]
as $n \to \infty$.

\begin{rem}
If we define an equivalence relation $\sim$ on the set of estimating
functions of the form \eqref{eq:pbef} by $G_n \sim H_n$ \emph{if and
  only if} $H_n = M_nG_n$ for an invertible $d \times d$-matrix $M_n$,
equivalent estimating functions yield identical estimators
$\hat{\theta}_n$. In particular, estimators obtained from equivalent
estimating functions share the same asymptotic properties. We freely
apply this property in the proofs of Section \ref{sec:asymptotics}.
\end{rem}

\subsection{Probabilistic notions}

Two notions from the theory of stochastic processes play a central role in this paper; the infinitesimal generator of a diffusion process $(X_t)$, and the dependence property known as $\rho$-mixing.

For general stochastic processes, mixing coefficients provide a way of measuring how dynamic dependence decays over time. Various notions appear in the literature and are often used to establish central limit theorems for processes that are not martingales; see e.g. \cite{mpe}.

A stationary Markov process $(X_t)$ is said to be \emph{$\rho$-mixing} if $\rho_X(t) \to 0$ as $t \to 0$, where
\begin{equation}\label{eq:rho}
\rho_X(t) = \rho\left(\sigma(X_0),\sigma(X_t)\right),
\end{equation}
with $\rho$ denoting correlation. A review of mixing properties for stationary Markov processes can be found in \cite{svhm}. Here easily checked conditions for $\rho$-mixing of one-dimensional diffusion processes are given.

With any weak solution of \eqref{eq:X} is associated a family of operators $(P_t^\theta)_{t \geq 0}$ where for $f \in \mathscr{L}^1(\mu_\theta)$,
\begin{equation*}
P_t^\theta f(x) = \ET\left(f(X_t) \given X_0=x\right).
\end{equation*}
Obviously, $P_t^\theta: \FSL \to \FSL$, and the semigroup property $P_t^\theta \circ P_s^\theta = P_{t+s}^\theta$ holds for all $t,s \geq 0$.

The \emph{(infinitesimal) generator} $\A_\theta$ of a diffusion $(X_t)$ is defined by
\begin{equation*}
\A_\theta f = \lim_{t \to 0}\frac{P_t^\theta f - f}{t},
\end{equation*}
whenever the limit $\A_\theta f$ exists in $\FSL$. Let $\D_{\A_\theta}$ denote the \emph{domain} of $\A_\theta$. For a weak solution of \eqref{eq:X} satisfying Condition 
\ref{cond:X} below, $\C^2_p(S) \subseteq \D_{\A_\theta}$, and for all $f \in \C^2_p(S)$
it holds that $\A_\theta f = \GT f$, where
\begin{equation}\label{eq:L}
\GT f(x) = a(x;\theta)\partial_xf(x) + \frac{1}{2}b^2(x;\theta)\partial_x^2f(x);
\end{equation}
see e.g. \cite{sef}.

Recall that $\lambda \in \R$ is an eigenvalue of $\A_\theta$ if
\begin{equation*}
\mathcal{A}_\theta f = \lambda f
\end{equation*}
for some $f \in \D_{\A_\theta}$. The collection of all eigenvalues is known as the \emph{spectrum} of $\mathcal{A}_\theta$ and will be denoted by $\mathscr{S}(\A_\theta)$. From spectral theory it is known that $\mathscr{S}(\A_\theta) \subset (-\infty,0]$. If $\mathscr{S}(\A_\theta) \subset (-\infty,-\lambda^*] \cup \{0\}$ for some $\lambda^*>0$, the generator $\A_\theta$ is said to have a \emph{spectral gap}. In particular, whenever the diffusion process $(X_t)$ is ergodic and reversible under $\PT$, the existence of a spectral gap $\lambda^*>0$ is equivalent to $(X_t)$ satisfying the $\rho$-mixing property; see \cite{svhm}.

\subsection{Assumptions}

To derive asymptotic results for diffusion models of the general form \eqref{eq:X}, we impose some mild dependence and regularity conditions on $(X_t)$. 

\begin{cond}\label{cond:X}
For any $\theta \in \Theta$, the stochastic differential equation
\begin{equation*}
dX_t = a(X_t;\theta) dt + b(X_t;\theta) dB_t, \hspace{0.2cm} X_0 \sim \mu_\theta
\end{equation*}
has a \emph{weak solution} $\left(\Omega,\F,(\F_t),\PT,(B_t),(X_t)\right)$ with the property that
\begin{itemize}
\myitem $(X_t)$ is stationary and $\rho$-mixing under $\PT$.
\end{itemize}
Moreover, the a priori given triplet $(a,b,\mu_\theta)$ satisfies the regularity conditions
\begin{itemize}
\myitem $a,b \in \C^{2,0}_p(S \times \Theta)$,
\myitem $\abs{a(x;\theta)} + \abs{b(x;\theta)}\leq_C 1+|x|$,
\myitem $b(x;\theta)>0$ for $x \in S$,
\myitem $\int_S |x|^k \mu_\theta(dx) < \infty$ for all $k \geq 1$.
\end{itemize}
\end{cond}

For the discretized filtration $\{\F_{t^n_i}\}$, we let $\F^n_i:=\F_{t^n_i}$ and the notation $\mu_0=\mu_{\theta_0}$, $\Pnull=\mathbb{P}_{\theta_0}$, etc., is applied throughout the paper. \\

The following condition on the true parameter value $\theta_0$ is
essential to the asymptotic theory. Here $\textnormal{int}(\Theta)$
denotes the interior of $\Theta$. 

\begin{cond}\label{cond:T}
The parameter $\theta \in \Theta \subset \R^d$, and it holds that
$\theta_0 \in \textnormal{int}(\Theta)$.
\end{cond}

	\section{Limit theory for discretized diffusions}\label{sec:limits}

This section is devoted to limit theorems for functionals of the form
\begin{equation}\label{eq:V}
V_n(f) = \frac{1}{n} \sum_{i=1}^n f(X_{t^n_{i-1}}),
\end{equation}
where $f$ takes values in $\R$ and $\{X_{t^n_i}\}_{i=0}^n$ is a
discretization of a diffusion process $(X_t)$ that satisfies Condition
\ref{cond:X}. 

First we state the law of large numbers, which follows from the
continuous-time ergodic theorem; see Corollary 10.9 in \cite{fmp}. 

\begin{lem}\label{lem:LLN}
Let $f \in \C^1_p(S)$. Then,
\begin{equation*}
V_n(f) \cPnull \mu_0(f).
\end{equation*}
\end{lem}

A central limit theorem requires stronger regularity assumptions on
$f$. In the following we define a suitable class of functions for our
purpose. Our only application of the central limit theorem in this
paper is to establish asymptotic normality of $G_n$-estimators in
Section \ref{sec:asymptotics}. Since $\ET\left(G_n(\theta)\right) =
0$, we can restrict attention to functions $f:S \to \R$ for which
$\mu_\theta(f)=0$ for the remainder of this section. 

The variance of the Gaussian limit distribution in the central limit
theorem involves the potential of the function $f$ with
$\mu_\theta(f)=0$. The {\it potential operator} is defined by
\begin{equation}\label{eq:U}
\U_\theta(f)(x) = \int_0^\infty P_t^\theta f(x) \dd t.
\end{equation}

To identify a (partial) domain for the operator $f \mapsto
\U_\theta(f)$, we use that the generator $\A_\theta$ of $(X_t)$ has a
spectral gap $\lambda>0$ under Condition \ref{cond:X}. This leads to a
well-known bound for the transition operator which we formulate as a
separate lemma. In the following, 
\begin{equation*}
\FSLnull = \left\{f:S \to \R : \mu_\theta(f^2)<\infty, \mu_\theta(f)=0\right\}.
\end{equation*}

\begin{lem}\label{lem:P}
Let $f \in \FSLnull$. Then under Condition \ref{cond:X} 
\begin{equation}\label{eq:P}
\norm{P_t^\theta f}_2 \leq e^{-\lambda t} \norm{f}_2
\end{equation}
for all $t \geq 0$.
\end{lem}

As a consequence, $\norm{\U_\theta(f)}_2<\infty$ for any $f \in
\FSLnull$, so the operator
\begin{equation*}
U_\theta: \FSLnull \to \FSL
\end{equation*}
is well-defined. It is obviously linear. General results on existence
and regularity implications of the potential $U_\theta(f)$ for
diffusion processes $(X_t)$ and $f:S \to \R$ can be found in
\cite{pe}. 

\vspace*{0.2cm}

For the central limit theorem, we restrict ourselves to the set of functions
\begin{equation}\label{eq:FSH}
\FSH = \left\{f \in \C^2_p(S) : \mu_\theta(f)=0,\U_\theta(f) \in \C^2_p(S)\right\},
\end{equation}
which ensures that $\A_\theta(U_\theta(f)) = \GT(U_\theta(f))$ and $\FSH \subset \FSLnull$. The following result characterizes the potential $U_\theta(f)$ as the solution of the so-called \emph{Poisson equation} for any $f \in \FSH$.

\vspace*{0.2cm}

\begin{prop}\label{prop:U}
Let $f \in \FSH$. Then, $\U_\theta(f)$ is a solution of the Poisson equation, i.e.
\begin{equation*}
\GT (\U_\theta(f)) = - f,
\end{equation*}
where $\GT$ is the generator  of $(X_t)$ given by the differential
operator \eqref{eq:L}. Moreover,
\begin{equation}\label{eq:U-bound}
\norm{U_\theta(f)}_2 \leq \lambda^{-1} \norm{f}_2.
\end{equation}
\end{prop}

\vspace*{0.2cm}

With Proposition \ref{prop:U} in place, we obtain the following
central limit theorem. Consistent with the general notation, we write $\U_0=\U_{\theta_0}$, $\mathscr{H}^2_0=\mathscr{H}^2_{\theta_0}$, etc., for the true parameter $\theta_0$.

\begin{prop}\label{prop:CLT}
Let $f \in \mathscr{H}^2_0$. If $n\Delta_n^3 \to 0$, then
\begin{equation*}
\sqrt{n\Delta_n} V_n(f) = \sqrt{n\Delta_n} \left(\frac{1}{n} \sum_{i=1}^n f(X_{t^n_{i-1}})\right)\cDnull \N\left(0,\V_0(f)\right),
\end{equation*}
where
\begin{equation}\label{eq:CLT.AVAR}
\V_0(f) = \mu_0\left([\partial_x U_0(f) b(\ccs;\theta_0)]^2\right) = 2\mu_0\left(f \U_0 (f)\right).
\end{equation}
\end{prop}

\begin{rem}
Compared to the low-frequency sampling scenario where $\Delta_n=\Delta>0$, the integral defining $U_\theta(f)$ in \eqref{eq:U} can be interpreted as the limit as $\Delta \to 0$ of the discrete-time potential,
\begin{equation*}
\tilde{U}_\theta(f) = \Delta \sum_{k=0}^\infty P_{k\Delta}^\theta f,
\end{equation*}
and the role of $U_0(f)$ in Proposition \ref{prop:CLT} is similar to that of $\tilde{U}_\theta(f)$ in the classic central limit theorem for functionals $\frac{1}{n} \sum_{i=1}^n f(X_{(i-1)\Delta})$; see e.g. Theorem~1, \cite{asd}.
\end{rem}

	\section{Asymptotic theory}\label{sec:asymptotics}

In this section we present our main asymptotic results for
prediction-based estimators. The main proof is based on general
asymptotic theory for estimating functions in \cite{jacso}; see also
\cite{sm}. For the most part, we restrict the discussion to estimating
functions of the form \eqref{eq:pbef} with $N=1$ and, for simplicity,
write 
\begin{equation}\label{eq:G}
G_n(\theta) = \sum_{i=q}^n \pi_{i-1} \left[f(X_{t^n_i})-\breve{\pi}_{i-1}(\theta)\right],
\end{equation}
$\mathcal{P}_{i-1}$ for the corresponding predictor spaces and so on
for objects in Section \ref{ssec:predicestfct} that depend on $j$. The
extension to multiple predictor functions $\{f_j\}_{j=1}^N$ is
considered in Section \ref{ssec:asymptotics.m}.

\subsection{Simple predictor spaces}\label{ssec:asymptotics.s}

The simplest class of estimating functions of the form \eqref{eq:G} is
obtained for $q=0$, in which case
$\mathcal{P}_{i-1}=\text{span}\{1\}$. The orthogonal projection is
$\breve{\pi}_{i-1}(\theta) = \mu_\theta(f)$, and the one-dimensional
predictor space $\mathcal{P}_{i-1}$ enables us to estimate one real parameter
$\theta \in \Theta \subseteq \R$. Therefore, we study the
one-dimensional estimating function  
\begin{equation}\label{eq:PBEF.s}
G_n(\theta) = \sum_{i=1}^n \left[f(X_{t^n_i})-\mu_\theta(f)\right].
\end{equation}
Such estimating functions were studied by \cite{sef}.

We easily identify conditions that ensure consistency and asymptotic
normality of $G_n$-estimators 
\begin{cond}\label{cond:G.s}
Suppose that
\begin{itemize}
\myitem $f^*(x) := f(x) - \mu_0(f) \in \mathscr{H} ^2_0$,
\myitem $\theta \mapsto \mu_\theta(f) \in \C^1$.
\end{itemize}
\end{cond}

\vspace{2mm}

\begin{thm}\label{thm:PBEF.s}
Assume Condition \ref{cond:G.s} and the identifiability condition
$\partial_{\theta} \mu_\theta(f) \neq 0$ for all
$\theta \in \Theta$. Define $\kappa(\theta) = \mu_\theta(f)$. Then the
following assertions hold. 
\begin{itemize}
\myitem There exists a consistent sequence of $G_n$-estimators
$(\hat{\theta}_n)$ which, as $n \to \infty$, is uniquely given by
$\hat{\theta}_n =  \kappa^{-1}\left( \frac1n \sum_{i=1}^n f(X_{t^n_i}) \right)$
with $\Pnull$-probability approaching one.
\myitem If, moreover, $n\Delta_n^3 \to 0$, then
\begin{equation}\label{eq:AN.s}
\sqrt{n\Delta_n}\left(\hat{\theta}_n-\theta_0\right) \cDnull \mathcal{N}\left(0,\left[\partial_{\theta}\mu_{\theta_0}(f)\right]^{-2} \V_0(f)\right),
\end{equation}
where $\V_0(f) = 2 \mu_0(f^* \U_0(f^*))$.
\end{itemize}
\end{thm}

\subsection{1-lag predictor spaces}\label{ssec:asymptotics.1}

The inclusion of past observations into the predictor space
$\mathcal{P}_{i-1}$ raises the mathematical complexity
dramatically. We show that for $q=1$, prediction-based
$G_n$-estimators remain consistent and asymptotically normal under
suitable regularity conditions. \\ 

For $q=1$, the basis vector $Z_{i-1}=(1,f(X_{t^n_{i-1}}))^T$, and it
follows from the normal equations \eqref{eq:moments} that
\begin{equation*}
\breve{\pi}_{i-1}(\theta) = \breve{a}_n(\theta)_0 + \breve{a}_n(\theta)_1 f(X_{t^n_{i-1}}),
\end{equation*}
where $\breve{a}_n(\theta)_0$ and $\breve{a}_n(\theta)_1$ are uniquely determined by the moment conditions
\begin{eqnarray*}
\breve{a}_n(\theta)_0 &=& \mu_\theta(f) \left(1-\breve{a}_n(\theta)_1\right), \\
\breve{a}_n(\theta)_1 &=& \frac{\ET\left[f(X_0)f(X_{\Delta_n})\right]-\left[\mu_\theta(f)\right]^2}{\VT f(X_0)},
\end{eqnarray*}
and consistent with a two-dimensional predictor space $\mathcal{P}_{i-1}$, we suppose that $d=2$ and study the estimating function
\begin{equation}\label{eq:PBEF.1}
G_n(\theta) = \sum_{i=1}^n \vtwo{1}{f(X_{t^n_{i-1}})} \left[f(X_{t^n_i})-\breve{a}_n(\theta)_0-\breve{a}_n(\theta)_1f(X_{t^n_{i-1}})\right].
\end{equation}

As part of the proof of Lemma \ref{lem:PBEF.1} below, we show that the
projection coefficient $\breve{a}_n(\theta)$ has an expansion
\begin{equation}\label{eq:a}
\breve{a}_n(\theta) = \vtwo{0}{1} + \Delta_n \vtwo{-K_f(\theta)\mu_\theta(f)}{K_f(\theta)} + \Delta_n^2 R(\Delta_n;\theta),
\end{equation}
where $|R(\Delta_n;\theta)| \leq C(\theta)$ and 
\begin{equation}\label{eq:K}
K_f(\theta) = \frac{\mu_\theta(f \GT f)}{\VT f(X_0)}.
\end{equation}
This observation enables us to formulate a set of regularity conditions on $G_n$ for the asymptotic theory:

\begin{cond}\label{cond:G.1}
Suppose that
\begin{itemize}
\myitem $f \in \C^4_p(S)$,
\myitem $f^*_1(x) = K_f(\theta_0)\left[\mu_0(f)-f(x)\right] \in \mathscr{H}^2_0$,
\myitem $f^*_2(x) = f(x) \left[ \Gnull f(x) - f^*_1(x) \right] \in \mathscr{H}^2_0$,
\myitem $(\theta \mapsto \mu_\theta(f)) \in \C^1$, $(\theta \mapsto
K_f(\theta)) \in \C^1$ and in \eqref{eq:a}
\begin{equation}\label{eq:dR}
\sup_{\theta \in \M} \norm{\partial_{\theta^T} R(\Delta_n;\theta)} \leq C(\M),
\end{equation}
for any compact subset $\M \subseteq \Theta$ and for
$\Delta_n$ sufficiently small. 
\end{itemize}
\end{cond}

\vspace*{0.2cm}

The matrix norm $\norm{\ccs}$ in \eqref{eq:dR} can be chosen
arbitrarily, and we suppose for convenience that $\norm{\ccs}$ is
submultiplicative. The following lemma essentially implies the
existence of a consistent sequence of $G_n$-estimators in Theorem \ref{thm:PBEF.1}. As the proof is somewhat long, we formulate it as a separate result.

\begin{lem}\label{lem:PBEF.1}
Assume that Condition \ref{cond:G.1} holds. Then, for any $\theta \in \Theta$,
\begin{equation}\label{eq:gamma}
(n\Delta_n)^{-1} G_n(\theta) \cPnull \gamma(\theta_0;\theta) = \vtwo{K_f(\theta) (\mu_\theta-\mu_0)(f)}{\mu_0(f \Gnull f) - K_f(\theta)\left[\mu_0(f^2) - \mu_0(f)\mu_\theta(f)\right]}
\end{equation}
and, moreover, for any compact subset $\M \subseteq \Theta$ 
\begin{equation}\label{eq:W}
\sup_{\theta \in \M} \norm{(n\Delta_n)^{-1}\partial_{\theta^T} G_n(\theta) - W(\theta)} \cPnull 0
\end{equation}
where
\begin{equation}
\label{eq:Wtheta}
W(\theta) =
\mtwo{1}{\mu_0(f)}{\mu_0(f)}{\mu_0(f^2)}
\mtwo{\partial_{\theta_1}\left[K_f(\theta)\mu_\theta(f)\right]}{\partial_{\theta_2}\left[K_f(\theta)\mu_\theta(f)\right]}{-\partial_{\theta_1} K_f(\theta)}{-\partial_{\theta_2} K_f(\theta)}.
\end{equation}
\end{lem}

\vspace*{0.2cm}

\begin{thm}\label{thm:PBEF.1}
Assume Condition \ref{cond:G.1} and suppose that $W(\theta_0)$ is
non-singular and that the following identifiability condition is satisfied
\begin{equation*}
\gamma(\theta_0;\theta) \neq 0 \ \ \mbox{ for all } \theta \neq \theta_0.
\end{equation*}
Then the following assertions hold:
\begin{itemize}
\myitem There exists a consistent sequence of $G_n$-estimators
$(\hat{\theta}_n)$, which is unique in any compact 
subset $\M \subseteq \Theta$ containing $\theta_0$ with
$\Pnull$-probability approaching one as $n \to \infty$.
\myitem If, moreover, $n\Delta_n^3 \to 0$, then
\begin{equation}\label{eq:AN.1}
\sqrt{n\Delta_n}\left(\hat{\theta}_n-\theta_0\right) \cDnull \mathcal{N}_2\left(0,\left[W(\theta_0)^{-1} \V_0(f) (W(\theta_0)^{-1})^T\right]\right),
\end{equation}
wheren $W(\theta_0))$ is given by \eqref{eq:Wtheta} and
\begin{eqnarray*}
\V_0(f)_{11} &=&  \mu_0\left(\left[\partial_x
\U_0(f^*_1)b(\ccs;\theta_0)\right]^2\right) = 2 \mu_0\left( f_1^*
                 \U_0(f^*_1)\right) \\ && \\
\V_0(f)_{12} &=&  \V_0(f)_{21} = \mu_0\left( \partial_x \U_0(f^*_1)\left[\partial_x
                 U_0(f^*_2) + f f'\right]
                 b^2(\ccs;\theta_0)\right) \\
&=& 
\mu_0\left( f_1^*  \U_0(f^*_2) + f_2^*  \U_0(f^*_1)  \right) +
   \mu_0\left( \partial_x \U_0(f^*_1) ff'  b^2(\ccs;\theta_0)\right)
  \\ && \\
\V_0(f)_{22} &=&  \mu_0\left(\left[\partial_x \U_0(f^*_2) +
                 f f' \right]^2 b^2(\ccs;\theta_0)\right) \\
&=& 2 \mu_0\left( f_2^* \U_0(f^*_2)\right) + \mu_0 \left( [ff'
    b(\ccs;\theta_0)]^2\right) + 2 \mu_0\left( \partial_x \U_0(f^*_2)
    ff'  b^2(\ccs;\theta_0)\right). 
\end{eqnarray*}

\end{itemize}
\end{thm}

\begin{rem}
If we denote the estimating function \eqref{eq:PBEF.1} as
\begin{equation*}
G_n(\theta) = \sum_{i=1}^n g(\Delta_n,X_{t^n_i},X_{t^n_{i-1}};\theta),
\end{equation*}
the proof of Lemma \ref{lem:PBEF.1} shows that
\begin{equation*}
\ET\left(g(\Delta_n,X_{t^n_i},X_{t^n_{i-1}};\theta) \given \F^n_{i-1}\right) = \Delta_n  g^*(X_{t^n_{i-1}};\theta) + \Delta_n^2 R(\Delta_n,X_{t^n_{i-1}};\theta)
\end{equation*}
for a non-zero function $g^*$ and $\theta \in \Theta$. Therefore, the
estimating functions in this section lie outside the class of
approximate martingale estimating functions defined in
\cite{eed-2017}. In particular, the proof of asymptotic normality in Theorem
\ref{thm:PBEF.1} requires extra work, because the remainder term
obtained by compensating $G_n$ is non-negligible. 
\end{rem}

\subsection{Multiple predictor functions and optimal estimation}\label{ssec:asymptotics.m}

Estimating functions with multiple predictor functions,
\begin{equation}\label{eq:G.m}
G_n(\theta) = \sum_{i=q}^n \sum_{j=1}^N \pi_{i-1,j} \left[f_j(X_{t^n_i})-\breve{\pi}_{i-1,j}(\theta)\right]
\end{equation}
appear frequently in practice. In the following, we indicate how to extend the asymptotic theory from estimating functions with a single predictor function \eqref{eq:G} to the more general case \eqref{eq:G.m} and briefly consider optimal estimation in relation to over-identification of the parameter $\theta \in \Theta \subset \R^d$. \\

To extend the proof in Appendix~A from estimating functions with a
single predictor function \eqref{eq:G} to estimating functions of the
more general form \eqref{eq:G.m}, we consider the more compact vector representation of the estimating functions \eqref{eq:G.m}
\begin{equation}\label{eq:PBEF.m}
G_n(\theta) = A_n(\theta) \sum_{i=q}^n Z_{i-1}\left[F(X_{t^n_i}) - \breve{\Pi}_{i-1}(\theta)\right],
\end{equation}
where $F(x)=\left(f_1(x),\ldots,f_N(x)\right)^T$, $\breve{\Pi}_{i-1}(\theta) = \left(\breve{\pi}_{i-1,1}(\theta),\ldots,\breve{\pi}_{i-1,N}(\theta)\right)^T$ and
\begin{equation}\label{eq:ZM}
Z_{i-1} = \left(
	\begin{array}{cccc}
		Z_{i-1,1} & 0_{q_1+1} & \cdots & 0_{q_1+1} \\ 0_{q_2+1} & Z_{i-1,2} & \cdots & 0_{q_2+1} \\ \vdots & \vdots & \ddots & \vdots \\ 0_{q_N+1} & 0_{q_N+1} & \cdots & Z_{i-1,N}
	\end{array}
\right).
\end{equation}
Recall that $Z_{i-1,j}$ denotes the column vector \eqref{eq:Z} of basis elements of $\mathcal{P}_{i-1,j}$ and the notation $0_{q_j+1}$ denotes a column vector of length $q_j+1$ containing zeroes only. Consistently, the dimension of $Z_{i-1}$ in \eqref{eq:ZM} is $\bar{d} \times N$ where $\bar{d}:= N + \sum_{j=1}^N q_j$. The coefficient matrix $A_n(\theta)$ is $d \times \bar{d}$ to match a $d$-dimensional parameter $\theta$. \\

To prove asymptotic results for the more general estimating equations,
impose the condition that $A_n(\theta) \to A(\theta)$ as $n \to
\infty$ and examine, by methods analogous to those used above, the normalized sum 
\begin{equation*}
V_n \times \sum_{i=q}^n Z_{i-1}\left[F(X_{t^n_i}) - \breve{\Pi}_{i-1}(\theta)\right],
\end{equation*}
where $V_n$ is a diagonal $\bar{d} \times \bar{d}$ matrix,
\begin{equation*}
V_n = \text{diag}\left(v_{n,1}^{(1)},\ldots,v_{n,q_1+1}^{(1)},\ldots,v_{n,1}^{(N)},\ldots,v_{n,q_N+1}^{(N)}\right),
\end{equation*}
and $v_{n,k_j}^{(j)} \to 0$ at appropriate rates, e.g. $v_{n,k_j}^{(j)}=n^{-1}$ or $v_{n,k_j}^{(j)} = (n\Delta_n)^{-1}$. \\

The condition $d \leq \bar{d}$ is necessary for $\theta$ to be
identified by the estimating equation $G_n(\theta)=0$, and we say that
$\theta$ is \emph{over-identified} if $d<\bar{d}$. Whereas $Z_{i-1}$,
$F$ and $\breve{\Pi}_{i-1}(\theta)$ are fully determined by our choice
of predictor functions $\{f_j\}_{j=1}^N$ and corresponding predictor
spaces $\{\mathcal{P}_{i-1,j}\}_{j}$, the coefficient matrix
$A_n(\theta)$ can be chosen optimally if $d<\bar{d}$, see 
\cite{oe} and \cite{pbef-2011}.

	\section{Estimating the asymptotic variance}\label{sec:avar}

Estimation of the asymptotic variance (AVAR) of $\hat{\theta}_n$ is
necessary for the construction of confidence intervals in practice. In
this section we propose a Monte Carlo-based method for calculating the
difficult parts of the asymptotic variance (or covariance matrix) for
the estimators derived in Sections \ref{ssec:asymptotics.s} and
\ref{ssec:asymptotics.1}. Moreover, we derive an upper bound for
$\AVT$ for estimating functions \eqref{eq:PBEF.s} and show that it is
exact when estimating the mean of an Ornstein-Uhlenbeck process.

Terms in the asymptotic variance that are integrals of known functions
with respect to the invariant measure can be found by standard
methods. The difficult parts of the asymptotic variance are integrals
with respect to the invariant measure that involve the potential. In
the expression for $\V_0(f)$ in Theorem \ref{thm:PBEF.1} there are
terms of the form $\mu_\theta (f_1 \partial_x U_\theta (f_2))$, where
$f_2 \in \FSH$. If we assume that the invariant measure $\mu_\theta$
has a density $\nu_\theta$ with respect to Lebesgue measure on the
state space $S = (\ell, r)$ ($-\infty \leq \ell < r \leq \infty)$,
which holds under weak regularity conditions, then it follows by
integration by parts that 
\begin{eqnarray*}
\mu_\theta \left( f_1 \partial_x U_\theta (f_2) \right) &=&
\int_\ell^rf_1(x) \nu_\theta (x) \partial_x U_\theta (f_2) (x) dx \\
&=& \nu_\theta (r) f_1(r) U_0(f_2)(r) - \nu_\theta (\ell) f_1(\ell)
    U_0(f_2)(\ell) -  \mu_\theta \left( U_0 (f_2) \left[ f_1' + f_1
    (\log \nu_\theta)'  \right] \right),
\end{eqnarray*}
where the function values at the end-points may have to be interpreted
as limits and will often be equal to zero. Now an inspection of the
expressions for the asymptotic variance in Theorems \ref{thm:PBEF.s} and
\ref{thm:PBEF.1} shows that all difficult terms are of the form $\mu_\theta
(g_1 U_\theta (g_2))$ with $g_1 \in \FSL$ and $g_2 \in \FSH$, and in
the following we propose a Monte Carlo method for calculating such
terms.  

For the construction we suppose that $\{T_i\}$ is a sequence of
independent random variables defined on an auxiliary probability space
$(\Omega',\F',\mathbb{P}'_\gamma)$ such that $T_i \sim \exp(\gamma)$
and consider the product extension 
\begin{equation*}
\tilde{\Omega}	= \Omega \times \Omega', \hspace*{0.5cm} \tilde{\F} = \F \otimes \F', \hspace*{0.5cm} \PTG = \PT \times \mathbb{P}'_\gamma.
\end{equation*}
Obviously, $\ETG f(X_0)) = \ET f(X_0)$ for any $f \in \mathscr{L}^1(\mu_\theta)$. 
For any $g_1 \in \FSL$ and $g_2 \in \FSH$, 
\begin{eqnarray*}
\mu_\theta\left(g_1\U_\theta(g_2)\right)
&=& \int_S g_1(x) \left(\int_0^\infty P_t^\theta g_2(x) \dd t\right) \mu_\theta(dx) \\
&=& \int_0^\infty \left(\int_S g_1(x) P_t^\theta g_2(x) \mu_\theta(dx)\right) \dd t \\
&=& \int_0^\infty \left(\int_S \ET\left(g_1(X_0)g_2(X_t) \given X_0=x\right) \mu_\theta(dx)\right) \dd t \\
&=& \int_0^\infty \ET\left(g_1(X_0)g_2(X_t)\right) \dd t \\
&=& \gamma^{-1} \int_0^\infty e^{\gamma t} \ETG\left(g_1(X_0)g_2(X_t)\right) \gamma e^{-\gamma t} \dd t \\
&=& \gamma^{-1} \int_0^\infty \ETG\left(e^{\gamma T_i} g_1(X_0)g_2(X_{T_i}) \given T_i=t\right) \gamma e^{-\gamma t} \dd t \\
&=& \gamma^{-1} \ETG\left[e^{\gamma T_i} g_1(X_0)g_2(X_{T_i})\right],
\end{eqnarray*}
where we have used Fubini's theorem and the fact that $(X_t)$ and $T_i$ are
independent on $\tilde{\Omega}$ under $\PTG$.

As a consequence, if $(X^{(i)}_t)$ are independent trajectories of
$(X_t)$ under $\PT$, the estimator 
\begin{equation}\label{eq:AVAR.est}
\gamma^{-1} \frac{1}{K} \sum_{i=1}^{K} e^{\gamma T_i} g_1\left(X^{(i)}_0\right) g_2\left(X^{(i)}_{T_i}\right)
\end{equation}
converges $\PTG$-almost surely to
$\mu_\theta\left(g_1\U_\theta(g_2)\right)$ as $K \to \infty$ for any
$g_1 \in \FSL$ and $g_2 \in \FSH$.

\subsection{Simple predictor spaces}\label{ssec:avar.s}

Let us consider the estimator in Section \ref{ssec:asymptotics.s} in
detail. By Theorem \ref{thm:PBEF.s}, 
\begin{equation}\label{eq:AVAR.s}
\AVT = \frac{2\mu_0\left(f^* \U_0(f^*)\right)}{[\partial_{\theta}\mu_0(f)]^2}
\end{equation}
with $f^*=f-\mu_0(f)$.
Thus the following algorithm can be used to estimate the asymptotic variance: \\

\begin{minipage}{14.2cm}
	\begin{framed}
		\textsc{Monte Carlo Estimation of $\AVT$}
		\emph{
		\begin{enumerate}
		\item	Determine $\hat{\theta}_n$,
		\item	Simulate $K$ independent variables $T_i \sim \exp(\gamma)$ for a fixed $\gamma>0$,
		\item	Simulate $K$ independent trajectories $t \mapsto X_t^{(i)}$ on $[0,T_i]$ under $\PTH$,
		\item	Evaluate
				\begin{equation}\label{eq:AVE}
				\AVE = 2 \cc [\partial_{\theta}\mu_{\hat{\theta}_n}(f)]^{-2} \gamma^{-1} \frac{1}{K} \sum_{i=1}^{K} e^{\gamma T_i} \hat{f}^*\left(X^{(i)}_0\right) \hat{f}^*\left(X^{(i)}_{T_i}\right),
				\end{equation}
				where $\hat{f}^*(x) := f(x)-\mu_{\hat{\theta}_n}(f)$.
		\end{enumerate}
		}
	\end{framed}
\end{minipage}

\vspace*{0.5cm}

In addition, the mixing property of $(X_t)$ leads to the following upper bound for $\AVT$.

\begin{prop}\label{prop:AVAR}
Suppose that $(X_t)$ and $G_n(\theta)$ satisfy \ref{cond:X} and \ref{cond:G.s}, respectively, and let $\lambda_0$ denote the spectral gap of $(X_t)$ under $\Pnull$. Then,
\begin{equation}\label{eq:AVAR.bound}
\AVT \leq \frac{2 \, \Vnull
  f(X_0)}{\lambda_0 \, [\partial_\theta\mu_{\theta_0}(f)]^2}. 
\end{equation}
\end{prop}

\begin{example}
The Ornstein-Uhlenbeck process
\begin{equation*}
dX_t = \kappa(\eta-X_t) dt + \xi dB_t,
\end{equation*}
with $\kappa, \xi > 0$ and $\eta \in \R$, satisfies Condition
\ref{cond:X}. The invariant distribution is
$\N\left(\eta,\frac{\xi^2}{2\kappa}\right)$. \\ 

Estimation of $\eta$ (with $\kappa$ and $\xi$ are known ) provides an
illustrative example where the upper bound of in \eqref{eq:AVAR.bound}
is attained. We choose $f(x) = x$, and by direct calculation,
\begin{eqnarray*}
\U_0(f^*)(x) = \int_0^\infty \left[\Enull\left(X_t - \eta_0 \given
  X_0=x\right) \right] \dd t = \int_0^\infty \left[xe^{-\kappa t} +
  \eta_0\left(1-e^{-\kappa t}\right) - \eta_0\right] \dd t  = \frac{(x-\eta_0)}{\kappa}.
\end{eqnarray*}
As a consequence,
\begin{equation}
\mu_0\left(f^* \U_0(f^*)\right) = \frac{1}{\kappa} \int_{\R} (x-\eta_0)^2 \mu_0(dx) = \frac{\xi^2}{2\kappa^2}
\end{equation}
and
\begin{equation*}
\AVT = \left(\frac{\xi}{\kappa}\right)^2.
\end{equation*}
The bound \eqref{eq:AVAR.bound} is attained because $\Vnull(X_0) =
\frac{\xi^2}{2\kappa}$ and $\lambda_0=\kappa$. 
\end{example}

\bibliographystyle{natbib}
\bibliography{ref.bib}

\begin{thebibliography}{}

\bibitem[{A\"{\i}t-Sahalia}(1996){A\"{\i}t-Sahalia}]{npd}
{A\"{\i}t-Sahalia}, Y. (1996).
\newblock Nonparametric pricing of interest rate derivative securities.
\newblock {\em Econometrica\/}, {\bf 64}(3), 527--560.

\bibitem[{A\"{\i}t-Sahalia}(2002){A\"{\i}t-Sahalia}]{mle}
{A\"{\i}t-Sahalia}, Y. (2002).
\newblock Maximum likelihood estimation of discretely sampled diffusions: A
  closed-form approximation approach.
\newblock {\em Econometrica\/}, {\bf 70}(1), 223--262.

\bibitem[{Bandi} and {Phillips}(2003){Bandi} and {Phillips}]{ned}
{Bandi}, F. and {Phillips}, P. (2003).
\newblock Fully nonparametric estimation of scalar diffusion models.
\newblock {\em Econometrica\/}, {\bf 71}(1), 241–283.

\bibitem[{Beskos} {\em et~al.}(2006){Beskos}, {Papaspiliopoulos}, {Roberts},
  and {Fearnhead}]{eel}
{Beskos}, A., {Papaspiliopoulos}, O., {Roberts}, G., and {Fearnhead}, P.
  (2006).
\newblock Exact and computationally efficient likelihood-based estimation for
  discretely observed diffusion processes (with discussion).
\newblock {\em Journal of the Royal Statistical Society\/}, {\bf 68}(3),
  333–382.

\bibitem[{Beskos} {\em et~al.}(2009){Beskos}, {Papaspiliopoulos}, and
  {Roberts}]{mce}
{Beskos}, A., {Papaspiliopoulos}, O., and {Roberts}, G. (2009).
\newblock Monte carlo maximum likelihood estimation for discretely observed
  diffusion processes.
\newblock {\em Annals of Statistics\/}, {\bf 37}(1), 223--245.

\bibitem[{Bibby} and {S\o rensen}(1995){Bibby} and {S\o rensen}]{mef-1995}
{Bibby}, B. and {S\o rensen}, M. (1995).
\newblock Martingale estimation functions for discretely observed diffusion
  processes.
\newblock {\em Bernoulli\/}, {\bf 1}(1/2), 17--39.

\bibitem[Bladt {\em et~al.}(2016)Bladt, Finch, and {S\o
  rensen}]{bladtfinchsorensen}
Bladt, M., Finch, S., and {S\o rensen}, M. (2016).
\newblock Simulation of multivariate diffusion bridges.
\newblock {\em J.R. Statist.\ Soc.\ B\/}, {\bf 78}, 343--369.

\bibitem[{Comte} {\em et~al.}(2007){Comte}, {Genon-Catalot}, and
  {Rozenholc}]{ped}
{Comte}, F., {Genon-Catalot}, V., and {Rozenholc}, Y. (2007).
\newblock Penalized nonparametric mean square estimation of the coefficients of
  diffusion processes.
\newblock {\em Bernoulli\/}, {\bf 13}(2), 514–543.

\bibitem[{Dacunha-Castelle} and {Florens-Zmirou}(1986){Dacunha-Castelle} and
  {Florens-Zmirou}]{edc-1986}
{Dacunha-Castelle}, D. and {Florens-Zmirou}, D. (1986).
\newblock Estimation of the coefficients of a diffusion from discrete
  observations.
\newblock {\em Stochastics\/}, {\bf 19}, 263--284.

\bibitem[{Ditlevsen} and {S\o rensen}(2004){Ditlevsen} and {S\o
  rensen}]{iid-2004}
{Ditlevsen}, S. and {S\o rensen}, M. (2004).
\newblock Inference for observations of integrated diffusion processes.
\newblock {\em Scandinavian Journal of Statistics\/}, {\bf 31}, 417--429.

\bibitem[{Doukhan}(1994){Doukhan}]{mpe}
{Doukhan}, P. (1994).
\newblock {\em Mixing: Properties and Examples\/}.
\newblock Lecture Notes in Statistics 85. Springer-Verlag.

\bibitem[{Elerian} {\em et~al.}(2001){Elerian}, {Chib}, and {Shephard}]{lid}
{Elerian}, O., {Chib}, S., and {Shephard}, N. (2001).
\newblock Likelihood inference for discretely observed nonlinear diffusions.
\newblock {\em Econometrica\/}, {\bf 69}(4), 959--993.

\bibitem[{Eraker}(2001){Eraker}]{dmf}
{Eraker}, B. (2001).
\newblock Mcmc analysis of diffusion models with application to finance.
\newblock {\em Journal of Business \& Economic Statistics\/}, {\bf 19}(2),
  177--191.

\bibitem[{Fan}(2005){Fan}]{nme}
{Fan}, J. (2005).
\newblock A selective overview of nonparametric methods in financial
  econometrics.
\newblock {\em Statistical Science\/}, {\bf 20}(4), 317--337.

\bibitem[{Florens-Zmirou}(1989){Florens-Zmirou}]{asd}
{Florens-Zmirou}, D. (1989).
\newblock Approximate discrete-time schemes for statistics of diffusion
  processes.
\newblock {\em Statistics\/}, {\bf 20}, 547--557.

\bibitem[{Florens-Zmirou}(1993){Florens-Zmirou}]{edo}
{Florens-Zmirou}, D. (1993).
\newblock On estimating the diffusion coefficient from discrete observations.
\newblock {\em Journal of Applied Probability\/}, {\bf 30}(4), 790--804.

\bibitem[{Genon-Catalot} and {Jacod}(1993){Genon-Catalot} and
  {Jacod}]{edc-1993}
{Genon-Catalot}, V. and {Jacod}, J. (1993).
\newblock On the estimation of the diffusion coefficient for multi-dimensional
  diffusion processes.
\newblock {\em Ann. Inst. Henri Poincaré\/}, {\bf 29}(1), 119--151.

\bibitem[{Genon-Catalot} {\em et~al.}(1992){Genon-Catalot}, {Larédo}, and
  {Picard}]{ewm}
{Genon-Catalot}, V., {Larédo}, C., and {Picard}, D. (1992).
\newblock Non-parametric estimation of the diffusion coefficient by wavelet
  methods.
\newblock {\em Scandinavian Journal of Statistics\/}, {\bf 19}(4), 317--335.

\bibitem[{Genon-Catalot} {\em et~al.}(2000){Genon-Catalot}, {Jeantheau}, and
  {Larédo}]{svhm}
{Genon-Catalot}, V., {Jeantheau}, T., and {Larédo}, C. (2000).
\newblock Stochastic volatility models as hidden markov models and statistical
  applications.
\newblock {\em Bernoulli\/}, {\bf 6}(6), 1051--1079.

\bibitem[{Gloter}(2000){Gloter}]{edc-2000}
{Gloter}, A. (2000).
\newblock Discrete sampling of an integrated diffusion process and parameter
  estimation of the diffusion coefficient.
\newblock {\em ESAIM: Probability and Statistics\/}, {\bf 4}, 205--227.

\bibitem[{Gobet} {\em et~al.}(2004){Gobet}, {Hoffmann}, and {Reiß}]{esd}
{Gobet}, E., {Hoffmann}, M., and {Reiß}, M. (2004).
\newblock Nonparametric estimation of scalar diffusions based on low frequency
  data.
\newblock {\em Annals of Statistics\/}, {\bf 32}(5), 2223–2253.

\bibitem[{Godambe} and {Heyde}(1987){Godambe} and {Heyde}]{oe}
{Godambe}, V. and {Heyde}, C. (1987).
\newblock Quasi-likelihood and optimal estimation.
\newblock {\em International Statistical Review\/}, {\bf 55}(3), 231–244.

\bibitem[Hall and Heyde(1980)Hall and Heyde]{hallheyde}
Hall, P. and Heyde, C.~C. (1980).
\newblock {\em Martingale Limit Theory and Its Applications\/}.
\newblock Academic Press, New York.

\bibitem[{Hansen} and {Scheinkman}(1995){Hansen} and {Scheinkman}]{mmp}
{Hansen}, L. and {Scheinkman}, J. (1995).
\newblock Back to the future: Generating moment implications for
  continuous-time markov processes.
\newblock {\em Econometrica\/}, {\bf 63}(4), 767--804.

\bibitem[{Hansen} {\em et~al.}(1998){Hansen}, {Scheinkman}, and {Touzi}]{smd}
{Hansen}, L., {Scheinkman}, J., and {Touzi}, N. (1998).
\newblock Spectral methods for identifying scalar diffusions.
\newblock {\em Journal of Econometrics\/}, {\bf 86}(1), 1--32.

\bibitem[{Hoffmann}(1999a){Hoffmann}]{aed}
{Hoffmann}, M. (1999a).
\newblock Adaptive estimation in diffusion processes.
\newblock {\em Stochastic Processes and their Applications\/}, {\bf 79}(1),
  135--163.

\bibitem[{Hoffmann}(1999b){Hoffmann}]{lpe}
{Hoffmann}, M. (1999b).
\newblock $l_p$ estimation of the diffusion coefficient.
\newblock {\em Bernoulli\/}, {\bf 5}(3), 447--481.

\bibitem[{Häusler} and {Luschgy}(2015){Häusler} and {Luschgy}]{slt}
{Häusler}, E. and {Luschgy}, H. (2015).
\newblock {\em Stable Convergence and Stable Limit Theorems\/}.
\newblock Springer.

\bibitem[{Jacod}(2000){Jacod}]{ked}
{Jacod}, J. (2000).
\newblock Non-parametric kernel estimation of the coefficient of a diffusion.
\newblock {\em Scandinavian Journal of Statistics\/}, {\bf 27}, 83--96.

\bibitem[{Jacod} and {S\o rensen}(2018){Jacod} and {S\o rensen}]{jacso}
{Jacod}, J. and {S\o rensen}, M. (2018).
\newblock A review of asymptotic theory of estimating functions.
\newblock {\em Statistical Inference for Stochastic Processes\/}, {\bf 21},
  415--434.

\bibitem[{Jakobsen} and {S\o rensen}(2017){Jakobsen} and {S\o rensen}]{ehf}
{Jakobsen}, N. and {S\o rensen}, M. (2017).
\newblock Efficient estimation for diffusions sampled at high frequency over a
  fixed time interval.
\newblock {\em Bernoulli\/}, {\bf 23}(3), 1874–1910.

\bibitem[{J\o rgensen}(2017){J\o rgensen}]{phd}
{J\o rgensen}, E. (2017).
\newblock {\em Diffusion Models Observed at High Frequency and Applications in
  Finance\/}.
\newblock Ph.D. thesis, Department of Mathematical Sciences, University of
  Copenhagen.

\bibitem[{Kallenberg}(2002){Kallenberg}]{fmp}
{Kallenberg}, O. (2002).
\newblock {\em Foundations of Modern Probability\/}.
\newblock Springer-Verlag.

\bibitem[{Kessler}(1997){Kessler}]{eed-1997}
{Kessler}, M. (1997).
\newblock Estimation of an ergodic diffusion from discrete observations.
\newblock {\em Scandinavian Journal of Statistics\/}, {\bf 24}, 211--229.

\bibitem[{Kessler}(2000){Kessler}]{sef}
{Kessler}, M. (2000).
\newblock Simple and explicit estimating functions for a discretely observed
  diffusion process.
\newblock {\em Scandinavian Journal of Statistics\/}, {\bf 27}, 65--82.

\bibitem[{Kessler} and {S\o rensen}(1999){Kessler} and {S\o rensen}]{mef-1999}
{Kessler}, M. and {S\o rensen}, M. (1999).
\newblock Estimating equations based on eigenfunctions for a discretely
  observed diffusion process.
\newblock {\em Bernoulli\/}, {\bf 5}(2), 299--314.

\bibitem[{Li}(2013){Li}]{cde}
{Li}, C. (2013).
\newblock Maximum-likelihood estimation for diffusion processes via closed-form
  density expansions.
\newblock {\em Annals of Statistics\/}, {\bf 41}(3), 1350–1380.

\bibitem[{Pardoux} and {Veretennikov}(2001){Pardoux} and {Veretennikov}]{pe}
{Pardoux}, E. and {Veretennikov}, A.~Y. (2001).
\newblock On the poisson equation and diffusion approximation. i.
\newblock {\em Annals of Probability\/}, {\bf 29}(3), 1061--1085.

\bibitem[{Renò}(2008){Renò}]{nev}
{Renò}, R. (2008).
\newblock Nonparametric estimation of the diffusion coefficient of stochastic
  volatility models.
\newblock {\em Econometric Theory\/}, {\bf 24}(5), 1174--1206.

\bibitem[{Roberts} and {Stramer}(2001){Roberts} and {Stramer}]{imh}
{Roberts}, G. and {Stramer}, O. (2001).
\newblock On inference for partially observed nonlinear diffusion models using
  the metropolis–hastings algorithm.
\newblock {\em Biometrika\/}, {\bf 88}(3), 603--621.

\bibitem[{Rudin}(1987){Rudin}]{rca}
{Rudin}, W. (1987).
\newblock {\em Real and Complex Analysis\/}.
\newblock McGraw-Hill.

\bibitem[{S\o rensen}(2000){S\o rensen}]{pbef-2000}
{S\o rensen}, M. (2000).
\newblock Prediction-based estimating functions.
\newblock {\em Econometrics Journal\/}, {\bf 3}, 123--147.

\bibitem[{S\o rensen}(2011){S\o rensen}]{pbef-2011}
{S\o rensen}, M. (2011).
\newblock Prediction-based estimating functions: review and new developments.
\newblock {\em Brazilian Journal of Probability and Statistics\/}, {\bf 25}(3),
  362--391.

\bibitem[{S\o rensen}(2012){S\o rensen}]{sm}
{S\o rensen}, M. (2012).
\newblock Estimating functions for diffusion-type processes.
\newblock In M.~{Kessler}, A.~{Lindner}, and M.~{S\o rensen}, editors, {\em
  Statistical Methods for Stochastic Differential Equations\/}, pages 1--107.
  CRC Press.

\bibitem[{S\o rensen}(2017){S\o rensen}]{eed-2017}
{S\o rensen}, M. (2017).
\newblock Efficient estimation for ergodic diffusions sampled at high
  frequency.
\newblock Working paper.

\bibitem[{Yoshida}(1992){Yoshida}]{edp}
{Yoshida}, N. (1992).
\newblock Estimation for diffusion processes from discrete observation.
\newblock {\em Journal of Multivariate Analysis\/}, {\bf 41}(2), 220--242.

\end{thebibliography}
	\appendix
	\section*{Appendix A: Proofs}\label{sec:app.proofs}

\begin{proof}[Proof of Lemma \ref{lem:P}]
The diffusion process $(X_t)$ is reversible under Condition
\ref{cond:X}, so by Theorem~2.4 and Theorem~2.6 in \cite{svhm} 
$\norm{P_t^\theta f}_2 \leq \rho_X(t) \norm{f}_2 =  e^{-\lambda t} \norm{f}_2$,
for any $f \in \FSLnull$ , where $\lambda>0$ denotes the spectral gap
of $\A_\theta$. 
\end{proof}

\vspace{1mm}

\begin{proof}[Proof of Proposition \ref{prop:U}]
Let $U_\theta^{(n)}(f)=\int_0^n P_t^\theta f \dd t$. By Property~P4 in \cite{mmp}, $U_\theta^{(n)}(f) \in \D_{\A_\theta}$ for all $n \in \mathbb{N}$ and
\begin{equation*}
\lim_{n \to \infty} \A_\theta \left(U_\theta^{(n)}(f)\right) = \lim_{n \to \infty} \left[P_n^\theta f-f\right] = - f,
\end{equation*}
where limits are w.r.t. $\norm{\ccs}_2$. The latter equality holds because
$\norm{P_n^\theta f}_2 \leq \norm{f}_2 e^{-\lambda n} \to 0$.

By Jensen's inequality, Fubini's theorem and Lemma, \ref{lem:P}
$U_\theta^{(n)}(f)$ converges to $U_\theta(f)$ in $\FSL$ as $n \to
\infty$:
\begin{eqnarray*}
				\norm{U_\theta(f)-U_\theta^{(n)}(f)}_2^2
&	=		&	\int_S \left(\int_0^\infty 1\{t \geq n\} \lambda^{-1} e^{\lambda t} P_t^\theta f(x) \lambda e^{-\lambda t} \dd t\right)^2 \mu_\theta(dx) \\
&	\leq		&	\int_S \left(\int_0^\infty 1\{t \geq n\} \lambda^{-2} e^{2\lambda t} \left(P_t^\theta f(x)\right)^2 \lambda e^{-\lambda t} \dd t\right) \mu_\theta(dx) \\
&	=		&	\lambda^{-1} \int_n^\infty e^{\lambda t} \cc || P_t^\theta f ||^2_2 \dd t \\
&	\leq		&	\lambda^{-1} \norm{f}^2_2 \int_n^\infty e^{-\lambda t} \dd t 
= \lambda^{-2} \norm{f}^2_2 e^{-\lambda n} \to 0.
\end{eqnarray*}
Taking $n=0$, we obtain \eqref{eq:U-bound}. Using that $\A_\theta$ is
closed and linear, we conclude that $\A_\theta\left(U_\theta(f)\right)
= \GT\left(U_\theta(f)\right) = -f$; see e.g. Property~P7,
\cite{mmp}. 
\end{proof}

\vspace{1mm}

\begin{proof}[Proof of Proposition \ref{prop:CLT}]
The proof is an application of the central limit theorem for martingales. 
For completeness and because we need to extend the result in a
non-standard way later, we give the proof. First, note that
\begin{eqnarray*}
\frac{1}{\sqrt{n\Delta_n}} \int_0^{n\Delta_n} f(X_s) \dd s
&=& \frac{1}{\sqrt{n\Delta_n}} \sum_{i=1}^n \int_{(i-1)\Delta_n}^{i\Delta_n} f(X_s) \dd s \\
&=& \frac{1}{\sqrt{n\Delta_n}} \sum_{i=1}^n \int_{(i-1)\Delta_n}^{i\Delta_n} \left[f(X_s) - f(X_{t^n_{i-1}})\right] \dd s + \sqrt{n\Delta_n} V_n(f),
\end{eqnarray*}
where we will show that
\begin{equation}\label{eq:CLT1}
\frac{1}{\sqrt{n\Delta_n}} \sum_{i=1}^n \int_{(i-1)\Delta_n}^{i\Delta_n} \left[f(X_s) - f(X_{t^n_{i-1}})\right] \dd s = o_{\Pnull}(1).
\end{equation}

With $A_i:=\int_{(i-1)\Delta_n}^{i\Delta_n} \left[f(X_s) - f(X_{t^n_{i-1}})\right] \dd s$, Fubini's theorem combined with \ref{lemx:generator} implies that
\begin{equation*}
\Enull\left(A_i \given \F^n_{i-1}\right) = \int_0^{\Delta_n} u \cc R(u,X_{t^n_{i-1}};\theta_0) \dd u \leq \Delta_n^2 F(X_{t^n_{i-1}};\theta_0)
\end{equation*}
for a generic function $F(x;\theta_0)$ of polynomial growth in
$x$. Since $n\Delta_n^3 \to 0$, it follows by Lemma \ref{lem:LLN} that
\begin{equation*}
\frac{1}{\sqrt{n\Delta_n}} \sum_{i=1}^n \Enull\left(A_i \given \F^n_{i-1}\right) \leq (n\Delta_n^3)^{1/2} \frac{1}{n} \sum_{i=1}^n F(X_{t^n_{i-1}};\theta_0) \cPnull 0.
\end{equation*}
Moreover, for all $k \geq 1$, Jensen's inequality implies that
\begin{eqnarray*}
|A_i|^k &=& \Delta_n^k \abs{\frac{1}{\Delta_n} \int_{(i-1)\Delta_n}^{i\Delta_n} \left[f(X_s) - f(X_{t^n_{i-1}})\right] \dd s}^k \\
& \leq & \Delta_n^{k-1} \int_{(i-1)\Delta_n}^{i\Delta_n} |f(X_s) - f(X_{t^n_{i-1}})|^k \dd s
\leq \Delta_n^k \sup_{u \in [0,\Delta_n]} |f(X_{t^n_{i-1}+u}) - f(X_{t^n_{i-1}})|^k,
\end{eqnarray*}
and, hence, by Lemma \ref{lemx:XX-bound},
\begin{eqnarray*}
\frac{1}{n\Delta_n} \sum_{i=1}^n \Enull\left(|A_i|^2 \given \F^n_{i-1}\right)
&	\leq		&	\Delta_n \frac{1}{n} \sum_{i=1}^n \Enull\left(\sup_{u \in [0,\Delta_n]} |f(X_{t^n_{i-1}+u}) - f(X_{t^n_{i-1}})|^2 \given \F^n_{i-1}\right) \\
&	=		&	\Delta_n^2 \frac{1}{n} \sum_{i=1}^n R(\Delta_n,X_{t^n_{i-1}};\theta_0) \cPnull 0.
\end{eqnarray*}
The conclusion \eqref{eq:CLT1} now follows from Lemma~9 in \cite{edc-1993}. \\

To apply the central limit theorem for martingales, note that Proposition
\ref{prop:U} and Itô's formula applied to $U_0(f)$ imply that 
\begin{eqnarray*}
\U_0(f)(X_t)
&=& \U_0(f)(X_0) + \int_0^t \Gnull (U_0(f))(X_s) \dd s + \int_0^t \partial_x U_0(f)(X_s)b(X_s;\theta_0) dB_s \\
&=& \U_0(f)(X_0) - \int_0^t f(X_s) \dd s + \int_0^t \partial_x U_0(f)(X_s)b(X_s;\theta_0) dB_s,
\end{eqnarray*}
so
\begin{equation}\label{eq:CLT2}
\frac{1}{\sqrt{n\Delta_n}} \int_0^{n\Delta_n} f(X_s) \dd s = \frac{1}{\sqrt{n\Delta_n}} \int_0^{n\Delta_n} \partial_x U_0(f)(X_s)b(X_s;\theta_0) dB_s + o_{\Pnull}(1).
\end{equation}

The stochastic integral is a true martingale under
$\Pnull$ and by the ergodic theorem 
\begin{equation*}
\frac{1}{n\Delta_n} \int_0^{n\Delta_n} \left[\partial_x U_0(f)(X_s)b(X_s;\theta_0)\right]^2 \dd s \cPnull \mu_0\left([\partial_x U_0(f) b(\ccs;\theta_0)]^2\right).
\end{equation*}

In conclusion,
\begin{eqnarray}
											\sqrt{n\Delta_n} V_n(f)
\label{eqn:CLT.equiv1}	
\label{eqn:CLT.equiv2}	&	=			&	\frac{1}{\sqrt{n\Delta_n}} \int_0^{n\Delta_n} \partial_x U_0(f)(X_s)b(X_s;\theta_0) dB_s + o_{\Pnull}(1) \\
						&	\cDnull		&	\N\left(0,\mu_0\left([\partial_x U_0(f) b(\ccs;\theta_0)]^2\right)\right), \nonumber
\end{eqnarray}
where convergence in law under $\Pnull$ follows from the
continuous-time martingale central limit theorem (e.g.\ Theorem~6.31
in \cite{slt}) or the central limit theorem for martingale arrays
(e.g.\ Theorem 3.2 in \cite{hallheyde}). The conditional Lyapunov
condition can be verified as in the proof of Theorem \ref{thm:PBEF.1}.

The alternative expression for the asymptotic variance $\V_0(f)$ in
\eqref{eq:CLT.AVAR} follows because with $g(x) = U_0(f)$ and $b_0(x) =
b(x;\theta_0)$ it follows
from Proposition \ref{prop:U} that
\[
2 \mu_0(fg) = - \mu_0\left(\Gnull(g^2)\right) +
\mu_0\left(b_0^2 \left[\frac12 (g^2)'' - g g'' \right] \right) =
\mu_0 \left((b_0g')^2 \right), 
\]
where we have used that $\mu_0(\Gnull(g^2)) = 0$, see e.g.\ \cite{mmp},
p. 774. 
\end{proof}

\vspace{1mm}

\begin{proof}[Proof of Theorem \ref{thm:PBEF.s}]
Under the conditions of theorem, the function $\kappa$ is 1-1, and
$\kappa^{-1}$ is continuous. By Lemma \ref{lem:LLN}, $V_n(f) \cPnull
\kappa(\theta_0)$ as $n \to \infty$. We have assumed that $\theta_0
\in {\rm int} \, \Theta$, so $\kappa(\theta_0) \in {\rm int} \, \kappa
(\Theta)$, and hence $\Pnull (V_n(f) \in \kappa (\Theta)) \to 1$ as $n
\to \infty$.

When $V_n(f) \in \kappa (\Theta)$, $\hat \theta_n = \kappa^{-1}
(V_n(f))$ is the unique $G_n$-estimator. When $V_n(f) \notin \kappa
(\Theta)$, we set $\hat \theta_n := \theta^*$ for some $\theta^* \in
\Theta$. Then $\hat \theta_n \cPnull \theta_0$ as $n \to \infty$, and
by a Taylor expansion
\[
\sqrt{n\Delta_n} \left( \hat \theta_n - \theta_0\right) =
\partial_\theta \kappa (\theta_0) \sqrt{n\Delta_n}V_n(f^*) + o_{\Pnull}(1),
\]
so \eqref{eq:AN.s} follows from Proposition \ref{prop:CLT}.
\end{proof}

\vspace{1mm}

\begin{proof}[Proof of Lemma \ref{lem:PBEF.1}]
To simplify the presentation, we define
\begin{equation}\label{eq:H}
H_n(\theta) = \frac{1}{n\Delta_n} \sum_{i=1}^n g(\Delta_n,X_{t^n_i},X_{t^n_{i-1}};\theta)
\end{equation}
where $g=(g_1,g_2)^T$ is given by
\begin{eqnarray}
\label{eqn:g1} g_1(\Delta_n,X_{t^n_i},X_{t^n_{i-1}};\theta) &=&
 f(X_{t^n_i})-\breve{a}_n(\theta)_0-\breve{a}_n(\theta)_1f(X_{t^n_{i-1}}) \\ 
\label{eqn:g2} g_2(\Delta_n,X_{t^n_i},X_{t^n_{i-1}};\theta) &=&
 f(X_{t^n_{i-1}}) g_1(\Delta_n,X_{t^n_i},X_{t^n_{i-1}};\theta). 
\end{eqnarray}

As a first step we verify the expansion \eqref{eq:a} of
$\breve{a}_n(\theta)$ in powers of $\Delta_n$. By Lemma \ref{lemx:generator},
\begin{equation*}
\ET\left(f(X_{\Delta_n}) \given \F_0\right) = f(X_0) + \Delta_n \GT f(X_0) + \Delta_n^2R(\Delta_n,X_0;\theta),
\end{equation*}
which implies that
\[
\ET\left[f(X_0)f(X_{\Delta_n})\right]
= \ET\left[f(X_0)\ET(f(X_{\Delta_n}) \given \F_0)\right] 
= \mu_\theta(f^2) + \Delta_n\mu_\theta(f \GT f) + \Delta_n^2 R(\Delta_n;\theta),
\]
where $\abs{R(\Delta_n;\theta)} \leq C(\theta)$ for a constant $C(\theta)>0$. This yields the $\Delta_n$-expansion
\begin{equation}\label{eq:a1}
\breve{a}_n(\theta)_1 = \frac{\ET\left[f(X_0)f(X_{\Delta_n})\right]-\left[\mu_\theta(f)\right]^2}{\VT f(X_0)} = 1 + \Delta_n K_f(\theta) + \Delta_n^2 R(\Delta_n;\theta),
\end{equation}
and, as a consequence,
\begin{equation}\label{eq:a0}
\breve{a}_n(\theta)_0 = -\Delta_n K_f(\theta)\mu_\theta(f) + \Delta_n^2 R(\Delta_n;\theta).
\end{equation}
This expansion of $\breve{a}_n(\theta)$ together with
\begin{equation*}
\Enull\left(f(X_{t^n_i}) \given \F^n_{i-1}\right) = f(X_{t^n_{i-1}}) + \Delta_n \Gnull f(X_{t^n_{i-1}}) + \Delta_n^2R(\Delta_n,X_{t^n_{i-1}};\theta_0)
\end{equation*}
imply that
\begin{eqnarray}\label{eqn:g1.cond}
&		& \Enull\left[g_1(\Delta_n,X_{t^n_i},X_{t^n_{i-1}};\theta) \given \F^n_{i-1}\right] \nonumber \\
&	=	& \Enull\left(f(X_{t^n_i}) \given \F^n_{i-1}\right) - \breve{a}_n(\theta)_0 - \breve{a}_n(\theta)_1f(X_{t^n_{i-1}}) \nonumber \\
&	=	& \Delta_n \left(\Gnull f(X_{t^n_{i-1}}) + K_f(\theta)\left[\mu_\theta(f) - f(X_{t^n_{i-1}})\right]\right) + \Delta_n^2 R_0(\Delta_n,X_{t^n_{i-1}};\theta).
\end{eqnarray}

Hence, by Lemma \ref{lem:LLN},
\begin{eqnarray*}
&				& \frac{1}{n\Delta_n} \sum_{i=1}^n \Enull\left[g_1(\Delta_n,X_{t^n_i},X_{t^n_{i-1}};\theta) \given \F^n_{i-1}\right] \\
&	=			& \frac{1}{n}\sum_{i=1}^n\Gnull f(X_{t^n_{i-1}})+K_f(\theta)\cc\frac{1}{n}\sum_{i=1}^n\left[\mu_\theta(f)-f(X_{t^n_{i-1}})\right]+\frac{\Delta_n}{n}\sum_{i=1}^nR_0(\Delta_n,X_{t^n_{i-1}};\theta) \\
&	\cPnull	& K_f(\theta) (\mu_\theta-\mu_0)(f),
\end{eqnarray*}
where the contribution from the first term vanishes because
$\mu_0(\Gnull f)=0$; see e.g.\ \cite{mmp}. 

To apply Lemma~9 in \cite{edc-1993}, it remains to show that
\begin{equation}\label{eq:g1}
\frac{1}{n^2\Delta^2_n} \sum_{i=1}^n \Enull\left[g^2_1(\Delta_n,X_{t^n_i},X_{t^n_{i-1}};\theta) \given \F^n_{i-1}\right] = o_{\Pnull}(1).
\end{equation}
From the expansions \eqref{eq:a1} and \eqref{eq:a0}, it follows that
\begin{equation*}
\breve{\pi}_{i-1}(\theta) = \breve{a}_n(\theta)_0 + \breve{a}_n(\theta)_1f(X_{t^n_{i-1}}) =  f(X_{t^n_{i-1}}) + \Delta_n R(\Delta_n,X_{t^n_{i-1}};\theta),
\end{equation*}
which, in turn, yields the decomposition
\begin{eqnarray}\label{eqn:g1.decomp}
\lefteqn{g^2_1(\Delta_n,X_{t^n_i},X_{t^n_{i-1}};\theta) =} \\ &&
\left[f(X_{t^n_i})-f(X_{t^n_{i-1}})\right]^2 + 2
  \left[f(X_{t^n_i})-f(X_{t^n_{i-1}})\right] \Delta_n
  R(\Delta_n,X_{t^n_{i-1}};\theta) + \Delta_n^2
  R(\Delta_n,X_{t^n_{i-1}};\theta). \nonumber
\end{eqnarray}
Lemma \ref{lemx:XX-bound} implies that
\begin{equation*}
\frac{1}{n^2\Delta_n^2} \sum_{i=1}^n \Enull\left[|f(X_{t^n_i})-f(X_{t^n_{i-1}})|^2 \given \F^n_{i-1}\right] = \frac{1}{n\Delta_n}\frac{1}{n}\sum_{i=1}^n R(\Delta_n,X_{t^n_{i-1}};\theta_0) \cPnull 0,
\end{equation*}
where we use that $n\Delta_n \to \infty$. Similarly,
\begin{equation*}
\frac{1}{n^2\Delta_n} \sum_{i=1}^n R(\Delta_n,X_{t^n_{i-1}};\theta)\Enull\left[|f(X_{t^n_i})-f(X_{t^n_{i-1}})| \given \F^n_{i-1}\right] = \frac{\Delta_n^{1/2}}{n\Delta_n}\frac{1}{n}\sum_{i=1}^n R_0(\Delta_n,X_{t^n_{i-1}};\theta) \cPnull 0,
\end{equation*}
and, finally,
\begin{equation*}
\frac{1}{n^2}\sum_{i=1}^n R(\Delta_n,X_{t^n_{i-1}};\theta) \cPnull 0,
\end{equation*}
which together implies \eqref{eq:g1}. Thus, by Lemma~9 in \cite{edc-1993},
\begin{equation*}
\frac{1}{n\Delta_n} \sum_{i=1}^n g_1(\Delta_n,X_{t^n_i},X_{t^n_{i-1}};\theta) \cPnull K_f(\theta) (\mu_\theta-\mu_0)(f).
\end{equation*}

Similarly for $g_2(\Delta_n,X_{t^n_i},X_{t^n_{i-1}};\theta)$, it follows easily from \eqref{eqn:g1.cond} that
\begin{eqnarray*}
&		& \Enull\left[g_2(\Delta_n,X_{t^n_i},X_{t^n_{i-1}};\theta) \given \F^n_{i-1}\right] \\
&	=	& \Delta_n \left(f(X_{t^n_{i-1}}) \Gnull f(X_{t^n_{i-1}}) - K_f(\theta)f(X_{t^n_{i-1}})\left[f(X_{t^n_{i-1}}) - \mu_\theta(f)\right]\right) + \Delta_n^2 R_0(\Delta_n,X_{t^n_{i-1}};\theta),
\end{eqnarray*}
and, hence,
\begin{equation*}
\frac{1}{n\Delta_n} \sum_{i=1}^n \Enull\left[g_2(\Delta_n,X_{t^n_i},X_{t^n_{i-1}};\theta) \given \F^n_{i-1}\right] \cPnull \mu_0(f \Gnull f) - K_f(\theta)\left[\mu_0(f^2) - \mu_0(f)\mu_\theta(f)\right].
\end{equation*}
Moreover, since $g^2_2(\Delta_n,X_{t^n_i},X_{t^n_{i-1}};\theta) = f^2(X_{t^n_{i-1}})g^2_1(\Delta_n,X_{t^n_i},X_{t^n_{i-1}};\theta)$, we easily see that
\begin{equation*}
\frac{1}{n^2\Delta^2_n} \sum_{i=1}^n \Enull\left[g^2_2(\Delta_n,X_{t^n_i},X_{t^n_{i-1}};\theta) \given \F^n_{i-1}\right] = o_{\Pnull}(1),
\end{equation*}
so the first conclusion of the lemma follows from Lemma~9 in \cite{edc-1993}.

To establish the limit of $\partial_{\theta^T} H_n(\theta)$, we write
\begin{equation*}
H_n(\theta) = \frac{1}{n\Delta_n}\sum_{i=1}^n Z_{i-1}\left[f(X_{t^n_i})-Z_{i-1}^T \breve{a}_n(\theta)\right],
\end{equation*}
which implies
\begin{equation*}
\partial_{\theta^T} H_n(\theta) = - \frac{1}{n\Delta_n}\sum_{i=1}^n Z_{i-1}Z_{i-1}^T \partial_{\theta^T} \breve{a}_n(\theta) = Z_n(f)A_n(\theta),
\end{equation*}
where $Z_n(f) := \frac{1}{n} \sum_{i=1}^n Z_{i-1}Z_{i-1}^T$ and
$A_n(\theta) := -\Delta_n^{-1} \partial_{\theta^T}
\breve{a}_n(\theta)$. By Lemma \ref{lem:LLN},
\begin{equation*}
Z_n(f) \cPnull Z(f) = \mtwo{1}{\mu_0(f)}{\mu_0(f)}{\mu_0(f^2)}
\end{equation*}
and applying the expansion \eqref{eq:a},
\begin{equation*}
A_n(\theta) = \partial_{\theta^T} \vtwo{K_f(\theta) \mu_\theta(f)}{-K_f(\theta)} + \Delta_n \partial_{\theta^T}R(\Delta_n;\theta) \to \partial_{\theta^T} \vtwo{K_f(\theta) \mu_\theta(f)}{-K_f(\theta)} =: A(\theta),
\end{equation*}
which holds under the regularity assumption \eqref{eq:dR}. Collecting our observations,
\begin{equation*}
\partial_{\theta^T} H_n(\theta) \cPnull Z(f)A(\theta) =
\mtwo{1}{\mu_0(f)}{\mu_0(f)}{\mu_0(f^2)}
\mtwo{\partial_{\theta_1}\left[K_f(\theta)\mu_\theta(f)\right]}{\partial_{\theta_2}\left[K_f(\theta)\mu_\theta(f)\right]}{-\partial_{\theta_1} K_f(\theta)}{-\partial_{\theta_2} K_f(\theta)}.
\end{equation*}

To argue that the convergence is uniform over a compact subset
$\M \subseteq \Theta$, note that
\begin{eqnarray*}
\norm{\partial_{\theta^T} H_n(\theta) - Z(f)A(\theta)}
\leq	 \norm{Z_n(f)[A_n(\theta)-A(\theta)]} + \norm{[Z_n(f)-Z(f)]A(\theta)}
\end{eqnarray*}
and, in particular,
\begin{equation*}
\sup_{\theta \in \M}\norm{\partial_{\theta^T} H_n(\theta) - Z(f)A(\theta)} \leq \norm{Z_n(f)} \sup_{\theta \in \M}\norm{A_n(\theta)-A(\theta)} + \norm{Z_n(f)-Z(f)} \sup_{\theta \in \M}\norm{A(\theta)}.
\end{equation*}
By continuity of norms, $\norm{Z_n(f)} \cPnull \norm{Z(f)}$ and
$\norm{Z_n(f)-Z(f)} = o_{\Pnull}(1)$, so \eqref{eq:W} follows by
observing that 
\begin{equation*}
\sup_{\theta \in \M}\norm{A_n(\theta)-A(\theta)} = \Delta_n
\sup_{\theta \in \M} \norm{\partial_{\theta^T} R(\Delta_n;\theta)}
\leq C(\M) \Delta_n \to 0
\end{equation*}
and using the continuity of $\theta \mapsto A(\theta)$.
\end{proof}

\vspace{1mm}

\begin{proof}[Proof of Theorem \ref{thm:PBEF.1}]
We continue with the notation \eqref{eq:H}-\eqref{eqn:g2} introduced
above. Existence of a consistent sequence of $G_n$-estimators
$(\hat{\theta}_n)$ follows from Theorem 2.5 in \cite{jacso}, because
the conclusions of Lemma \ref{lem:PBEF.1} and the assumption that $W(\theta_0)$ is
non-singular imply Condition 2.2 in \cite{jacso}. The uniqueness
result follows from Theo\-rem 2.7 in \cite{jacso} under the
identifiability condition $\gamma(\theta_0;\theta) \neq 0$ for $\theta
\neq \theta_0$. The function $\theta \mapsto \gamma(\theta_0;\theta)$
is called $G(\theta)$ in \cite{jacso} and is necessarily continuous.

\vspace{2mm}

Asymptotic normality when $n\Delta_n^3 \to 0$ follows from Theorem
2.11 in \cite{jacso}. We only need to check that
\begin{equation}
\label{eq:nyclt}
\sqrt{n\Delta_n} H_n(\theta_0) \cDnull \N_2(0,\V_0(f)).
\end{equation}

We apply the Cram{\'e}r-Wold device to prove this weak convergence
result, i.e.\ we must prove that for all $c_1,c_2 \in \R$
\begin{eqnarray}
\label{eq:CW}
C_n &=&  c_1 \cc \frac{1}{\sqrt{n\Delta_n}} \sum_{i=1}^n
        g_1(\Delta_n,X_{t^n_i},X_{t^n_{i-1}};\theta_0) + c_2 \cc
        \frac{1}{\sqrt{n\Delta_n}} \sum_{i=1}^n
        g_2(\Delta_n,X_{t^n_i},X_{t^n_{i-1}};\theta_0) \\ 
&& \hspace{40mm} \cDnull		 \N\left(0,\mu_0\left(\left[\partial_x
   \U_0(c_1f^*_1+c_2f^*_2) + c_2 f f' \right]^2
   b^2(\ccs;\theta_0)\right)\right) \nonumber
\end{eqnarray}

Reusing the expansions \eqref{eq:a1} and \eqref{eq:a0}, we find that
\begin{equation*}
g_1(\Delta_n,X_{t^n_i},X_{t^n_{i-1}};\theta_0) =  
f(X_{t^n_i}) - f(X_{t^n_{i-1}}) + \Delta_n f^*_1(X_{t^n_{i-1}}) + \Delta_n^2 R(\Delta_n,X_{t^n_{i-1}};\theta_0)
\end{equation*}
where $f^*_1$ is defined in Condition \ref{cond:G.1}. Hence,
\begin{eqnarray*}
\lefteqn{\frac{1}{\sqrt{n\Delta_n}} \sum_{i=1}^n
  g_1(\Delta_n,X_{t^n_i},X_{t^n_{i-1}};\theta_0)} \\ & =& 
\frac{1}{\sqrt{n\Delta_n}} \sum_{i=1}^n \left[f(X_{t^n_i}) -
  f(X_{t^n_{i-1}})\right] + \sqrt{n\Delta_n} \cc V_n(f^*_1) +
  (n\Delta_n^3)^{1/2} \cc \frac{1}{n} \sum_{i=1}^n
  R(\Delta_n,X_{t^n_{i-1}};\theta_0) \\
&=&	\sqrt{n\Delta_n} \cc V_n(f^*_1) + o_{\Pnull}(1),
\end{eqnarray*}
because the first term in the expansion is a telescoping sum. Note
that asymptotic normality for the first coordinate of the estimating
function follows from Proposition \ref{prop:CLT}. However, to obtain joint
weak convegence, we need to consider the second coordinate too, which 
requires more work.

By Itô's formula,
\begin{equation*}
f(X_{t^n_i}) - f(X_{t^n_{i-1}}) = \Delta_n \Gnull f(X_{t^n_{i-1}}) + A_i(\theta_0) + M_i(\theta_0),
\end{equation*}
where
\begin{eqnarray*}
A_i(\theta)	&=& \int_{(i-1)\Delta_n}^{i\Delta_n} \left[\GT f(X_s) - \GT f(X_{t^n_{i-1}})\right] \dd s, \\
M_i(\theta)	&=& \int_{(i-1)\Delta_n}^{i\Delta_n} f'(X_s)b(X_s;\theta) dB_s,
\end{eqnarray*}
and, hence, by applying the expansions \eqref{eq:a1} and \eqref{eq:a0} as above,
\begin{eqnarray*}
\lefteqn{g_2(\Delta_n,X_{t^n_i},X_{t^n_{i-1}};\theta_0) =} \\ &&
f(X_{t^n_{i-1}}) A_i(\theta_0) + \Delta_n f^*_2(X_{t^n_{i-1}}) + f(X_{t^n_{i-1}}) M_i(\theta_0) + \Delta_n^2 R(\Delta_n,X_{t^n_{i-1}};\theta_0).
\end{eqnarray*}

A straightforward extension of the proof of \eqref{eq:CLT1} implies that
\begin{equation*}
\frac{1}{\sqrt{n\Delta_n}} \sum_{i=1}^n f(X_{t^n_{i-1}}) A_i(\theta_0) = o_{\Pnull}(1)
\end{equation*}
since $n\Delta_n^3 \to 0$ and, as a consequence,
\begin{equation*}
C_n =  \sqrt{n\Delta_n} V_n(f^*) + \frac{1}{\sqrt{n\Delta_n}} \sum_{i=1}^n f(X_{t^n_{i-1}}) M^n_i(\theta_0) + o_{\Pnull}(1),
\end{equation*}
where $f^* = c_1 f_1^* + c_2 f_2^*$.

To gather the non-negligible terms, we argue as in \eqref{eqn:CLT.equiv2} that
\[
\sqrt{n\Delta_n} V_n(f^*) 
= \frac{1}{\sqrt{n\Delta_n}} \sum_{i=1}^n
\int_{(i-1)\Delta_n}^{i\Delta_n} \partial_x
\U_0(f^*)(X_s)b(X_s;\theta_0)dB_s + o_{\Pnull}(1), 
\]
which, in turn, yields the stochastic integral representation
\[
C_n = \frac{1}{\sqrt{n\Delta_n}} \sum_{i=1}^n
\int_{(i-1)\Delta_n}^{i\Delta_n} \left[\partial_x \U_0(f^*)(X_s) +
  f(X_{t^n_{i-1}}) f'(X_s)\right] b(X_s;\theta_0) dB_s +
o_{\Pnull}(1). 
\]

At this point, we can apply the central limit theorem for martingale
difference arrays; see e.g.\ \cite{hallheyde} or \cite{slt}. To
shorten notation in the following, we define 
\begin{equation*}
Z_i := \int_{(i-1)\Delta_n}^{i\Delta_n} \left[\partial_x
  \U_0(f^*)(X_s) + f(X_{t^n_{i-1}}) f'(X_s)\right] b(X_s;\theta_0) dB_s,
\end{equation*}
and
\begin{equation*}
h(x) = \left[\partial_x \U_0(f^*)(x) + f(x) f'(x)\right]^2 b^2(x;\theta_0).
\end{equation*}

First, by the conditional Itô isometry, Tonelli's theorem and Lemma \ref{lemx:generator},
\begin{eqnarray*}
\lefteqn{\frac{1}{n\Delta_n}\sum_{i=1}^n \Enull\left((Z_i)^2 \given \F^n_{i-1}\right)} \\
&	=			& \frac{1}{n\Delta_n}\sum_{i=1}^n
                                  \Enull\left(\int_{(i-1)\Delta_n}^{i\Delta_n}
                                  \left[\partial_x \U_0(f^*)(X_s) +
                                  f(X_{t^n_{i-1}}) f'(X_s)\right]^2
                                  b^2(X_s;\theta_0) \dd s \given
                                  \F^n_{i-1}\right) \\ 
&	=			& \frac{1}{n\Delta_n}\sum_{i=1}^n
                                  \int_{(i-1)\Delta_n}^{i\Delta_n}
                                  \Enull\left(\left[\partial_x
                                  \U_0(f^*)(X_s) + f(X_{t^n_{i-1}})
                                  f'(X_s)\right]^2 b^2(X_s;\theta_0)
                                  \given \F^n_{i-1}\right) \dd s \\ 
&	=			& \frac{1}{n\Delta_n}\sum_{i=1}^n \int_0^{\Delta_n} \left[h(X_{t^n_{i-1}}) + u \cc R(u,X_{t^n_{i-1}};\theta_0)\right] \dd u \\
&	=			& \frac{1}{n} \sum_{i=1}^n h(X_{t^n_{i-1}}) + o_{\Pnull}(1) \\
&	\cPnull	& \mu_0\left(\left[\partial_x \U_0(f^*) + f f' \right]^2 b^2(\ccs;\theta_0)\right).
\end{eqnarray*}

Moreover, for any $g \in \C^2_p(S)$ and $k \geq 2$, the
Burkholder-Davis-Gundy inequality, Jensen's inequality, Tonelli's
theorem and Lemma \ref{lemx:generator}, respectively, imply that
\begin{eqnarray*}
					\Enull \left(\abs{\int_{(i-1)\Delta_n}^{i\Delta_n} g(X_s) dB_s}^k \given \F^n_{i-1}\right)
&	\leq		&	\Enull \left(\left(\int_{(i-1)\Delta_n}^{i\Delta_n} g^2(X_s) \dd s\right)^{k/2} \given \F^n_{i-1}\right) \\
&	\leq		&	\Delta_n^{k/2-1} \cc \Enull \left(\int_{(i-1)\Delta_n}^{i\Delta_n} |g(X_s)|^k \dd s \given \F^n_{i-1}\right) \\
&	=			&	\Delta_n^{k/2-1} \cc \int_{(i-1)\Delta_n}^{i\Delta_n} \Enull\left(|g(X_s)|^k \given \F^n_{i-1}\right) \dd s  \\
&	=			&	\Delta_n^{k/2-1} \cc \int_{0}^{\Delta_n} \left(|g(X_{t^n_{i-1}})|^k + u \cc R(u,X_{t^n_{i-1}};\theta_0)\right) \dd u  \\
&	\leq	&	\Delta_n^{k/2} |g(X_{t^n_{i-1}})|^k + \Delta_n^{k/2+1} F(X_{t^n_{i-1}};\theta_0),
\end{eqnarray*}
so based on the inequality
\begin{equation*}
\abs{Z_i}^3 \leq_C \abs{\int_{(i-1)\Delta_n}^{i\Delta_n} \partial_x \U_0(f^*)(X_s)b(X_s;\theta_0) dB_s}^3 + |f(X_{t^n_{i-1}})|^3\abs{\int_{(i-1)\Delta_n}^{i\Delta_n} f'(X_s)b(X_s;\theta_0) dB_s}^3,
\end{equation*}
we conclude that
\begin{eqnarray*}
\lefteqn{\frac{1}{(n\Delta_n)^{3/2}} \sum_{i=1}^n \Enull\left(|Z_i|^3
  \given \F^n_{i-1}\right) \leq_C} \\ 
&& \frac{1}{\sqrt{n}} \frac{1}{n} \sum_{i=1}^n \left[|\partial_x
   \U_0(f^*)(X_{t^n_{i-1}})|^3 + |f(X_{t^n_{i-1}}) f'(X_{t^n_{i-1}})|^3\right]|b(X_{t^n_{i-1}};\theta_0)|^3 \\ 
&& \hspace{25mm} \mbox{} + \frac{\Delta_n}{\sqrt{n}} \frac{1}{n} \sum_{i=1}^n
   F(X_{t^n_{i-1}};\theta_0) \cPnull  0.
\end{eqnarray*}

Now the martingale central limit theorem for triangular arrays implies
\eqref{eq:CW}, so \eqref{eq:nyclt} follows by the Cram{\'e}r-Wold
device.
The alternative expressions for the matrix $\V_0(f)$ follows because
by Proposition \ref{prop:CLT} $\mu_0\left([\partial_x U_0(g)
  b(\ccs;\theta_0)]^2\right) = 2\mu_0\left(g \U_0 (g)\right)$ for $g \in
\mathscr{H}^2_0$, and because with $g_i(x) = U_0(f_i^*)$ and
$b_0(x) = b(x;\theta_0)$ it follows from Proposition \ref{prop:U} that
\[
\mu_0(f_1^*g_2 + f_2^*g_1) = - \mu_0\left(\Gnull(g_1g_2)\right) +
\mu_0\left(b_0^2 g_1' g_2' \right) = \mu_0 \left( b_0^2g_1'g_2' \right), 
\]
where we have used that $\mu_0(\Gnull(g_1g_2)) = 0$, see e.g.\ \cite{mmp},
p. 774. 
 
\end{proof}

\vspace{1mm}

\begin{proof}[Proof of Proposition \ref{prop:AVAR}]
By the Cauchy-Schwarz inequality and the inequality \eqref{eq:U-bound}
\begin{equation*}
\abs{\mu_0\left(f^* \U_0(f^*)\right)} \leq \norm{f^*}_2
\norm{\U_0(f^*)}_2 \leq \frac{\norm{f^*}_2^2}{\lambda_0}, 
\end{equation*}
where $\lambda_0>0$ denotes the spectral gap of $(X_t)$ under
$\Pnull$. Hence,
\begin{equation*}
\AVT = \frac{2\, \mu_0\left(f^*
    \U_0(f^*)\right)}{[\partial_{\theta}\mu_{\theta_0}(f)]^2} \leq
\frac{2 \, \Vnull f(X_0)}{\lambda_0 \, [\partial_{\theta}\mu_{\theta_0}(f)]^2}. 
\end{equation*}
\end{proof}

	\section*{Appendix B: Moment expansions}\label{sec:app.expansions}

The proofs in Appendix~A rely on conditional moment expansions for
diffusion models and the following results are essentially taken from
\cite{edc-2000} and \cite{asd}, respectively. In the sequel, $\theta
\in \Theta$ is arbitrary and we assume for convenience that
$0<\Delta<1$. 

\begin{lemxB}\label{lemx:XX-bound}
Let $f \in \C^1_p(S)$. For any $k \geq 1$, there exists a constant
$C_{k,\theta}>0$ such that
\begin{equation*}
\ET\left(\sup_{s \in [0,\Delta]}|f(X_{t+s})-f(X_t)|^k \given
  \F_t\right) \leq C_{k,\theta} \Delta^{k/2} \left(1+|X_t|\right)^{C_{k,\theta}}.
\end{equation*}
\end{lemxB}

For completeness, we give a rough proof of the following theorem.

\begin{lemxB}\label{lemx:generator}
Suppose that $a(x;\theta) \in \C_p^{2k,0}(S \times \Theta)$, $b(x;\theta) \in \C_p^{2k,0}(S \times \Theta)$ and $f \in \C^{2(k+1)}_p(S)$ for some $k \geq 0$. Then,
\begin{equation*}
\ET\left(f(X_{t+\Delta}) \given \F_t\right) = \sum_{i=0}^k \frac{\Delta^i}{i!} \GT^i f(X_t) +  \Delta^{k+1}R(\Delta,X_t;\theta).
\end{equation*}
\end{lemxB}

\begin{proof}
We only consider $k=0$, the general case may be shown by induction; see Lemma~1.10, \cite{sm}. By Itô's formula,
\begin{equation*}
f(X_{t+\Delta}) = f(X_t) + \int_t^{t+\Delta} \GT f(X_s) \dd s + \int_t^{t+\Delta} \partial_x f(X_s)b(X_s;\theta)dB_s,
\end{equation*}
and since $\partial_x f$ and $b(\ccs;\theta)$ are of polynomial, respectively linear, growth in $x$, the stochastic integral is a true $(\F_t)$-martingale w.r.t. $\PT$ and
\begin{equation*}
\ET\left(f(X_{t+\Delta}) \given \F_t\right)= f(X_t) + \int_0^{\Delta} \ET\left(\GT f(X_{t+u}) \given \F_t\right) \dd u.
\end{equation*}
Moreover, since $\GT f$ is of polynomial growth in $x$,
\begin{equation*}
|\GT f(X_{t+u})| \leq_C 1 + |X_t|^C + |X_{t+u}-X_t|^C
\end{equation*}
and, hence,
\begin{equation*}
\Delta^{-1} \int_0^{\Delta} \ET\left(\GT f(X_{t+u}) \given \F_t\right) \dd u = R(\Delta,X_t;\theta),
\end{equation*}
by a simple application of Lemma \ref{lemx:XX-bound}.
\end{proof}

\end{document}